\documentclass[12pt]{article}
\usepackage{mathrsfs}
\usepackage{amsfonts}
\usepackage{amsmath,amssymb,amsfonts,amsthm,fancyhdr,mathrsfs}
\usepackage{bbm}
\usepackage{verbatim}
\usepackage{color}
\usepackage{dsfont}
\usepackage[left=2.5cm,right=2.5cm,top=3cm,bottom=3cm]{geometry}

\usepackage{latexsym}
\usepackage{graphicx}
\usepackage{enumerate}
\usepackage[colorlinks=true,urlcolor=blue,
citecolor=red,linkcolor=blue,linktocpage,pdfpagelabels,
bookmarksnumbered,bookmarksopen]{hyperref}
\usepackage{cite}
\usepackage{indentfirst}
\usepackage{amsmath}
\usepackage{booktabs}
\usepackage{threeparttable}
\usepackage{float}
\allowdisplaybreaks[4]
\numberwithin{equation}{section} 

\numberwithin{equation}{section}
\newtheorem{thm}{Theorem}[section]

\newtheorem{pro}[thm]{Proposition}
\newtheorem{lem}[thm]{Lemma}

\newtheorem{re}{Remark}[section]
\newenvironment{pf}{{\noindent \it \bf Proof:}}{{\hfill $\square$}\\}
\usepackage[pagewise]{lineno}

\title{Thermodynamic limit and $L^\infty$-convergence rate for the cubic-quintic Schr\"{o}dinger model}
\author{ \bf
	{Deke Li$^1$, Yuan Li$^1$ and  Qingxuan Wang$^2$\thanks{Corresponding author: Qingxuan Wang.}}\\
{\small \emph{$^1$School of Mathematics and Statistics, Lanzhou University, Lanzhou 730000, China}}\\
   {\small \emph{$^2$School of Mathematical Sciences, Zhejiang Normal University,	Jinhua, 321004, China}}
}
\date{}

\begin{document}
\maketitle
\footnote{E-mail: lidk2024@lzu.edu.cn (D. Li), li\_yuan@lzu.edu.cn (Y. Li), wangqx@zjnu.edu.cn (Q. Wang)}
\begin{abstract}
We investigate the thermodynamic limit for the cubic-quintic Schr\"{o}dinger model as the size of the domain tends to infinity with fixed density $\rho= N/|\mathcal{D}|$, where $N$ denotes particle number  and $|\mathcal{D}|$ denotes the volume of the bounded domain $\mathcal{D}\subset\mathbb{R}^d$ ($d=1,2,3$). We firstly prove the existence of  thermodynamic limit, which is equal to $-\frac{3}{32}$ for \(0<\rho\leq \frac{3}{4}\), while  $-\left(\frac{1}{2}-\frac{\rho}{3}\right)\frac{\rho}{2}$ for $\frac{3}{4}< \rho\leq 1$.  When \(0<\rho<1\) and \(\mathcal{D}\) is a spherical domain, we further show that, up to a scaling, the ground state of the cubic-quintic  Schr\"{o}dinger energy will converge strongly to a Thomas-Fermi ground state  in $L^2\cap L^6$. Finally, we obtain the $L^\infty$-convergence rate of ground states for \(0<\rho<3/4\) by developing a novel method, including some iterative techniques, uniform energy estimates and  gradient estimates. We believe this method is applicable to other general nonlinearities. 
\vspace{0.15cm}

\end{abstract}
\tableofcontents

\section{Introduction and main results}
In this paper, we consider the cubic-quintic Schr\"odinger equation
\begin{align}\label{1}
i\partial_t\Psi=-\Delta\Psi-|\Psi|^2\Psi+|\Psi|^4\Psi,\quad (x,t)\in \mathcal{D}\times\mathbb{R}^+,
\end{align}
where $\mathcal{D} \subset \mathbb{R}^{d}$  ($d=1,2,3$) denotes a domain. 
The cubic term (also known as Kerr nonlinearity \cite{PREdesvataikov}) with a negative coefficient (the focusing case) is a very natural case in physics. In particular, the incorporation of the defocusing quintic term is motivated by the stabilization of two-dimensional or three-dimensional vortex solitons \cite{PDmalomed}. This kind of model can be used to describe nonlinear optics,  mean-field theory of superconductivity, the motion of Bose-Einstein condensates and Langmuir waves in plasma physics, see \cite{PREdesvataikov,PREchen,PREmih,JMCSbuslaev} and the references therein.

Equation \eqref{1} has attracted  a great deal of attention, one can see \cite{Nguyen-Ricaud24,Nguyen-Ricaud25,RMPar,ARMAkillip,IMMartel,Qwang,lw-24,Soave-jde,Soave-jfa} and the references therein. In two and three dimensional, an effect of the quintic term is to prevent finite time blow-up which may occur in the purely cubic case (cf. \cite{ANYcazen}). Indeed, from the conservation of the energy and the mass, combining H\"older’s inequality,
\begin{align*}
   \|\Psi\|^4_{L^4(\mathcal{D})}\leq\|\Psi\|_{L^2(\mathcal{D})}\|\Psi\|^3_{L^6(\mathcal{D})},
\end{align*}
one can see that the cubic focusing part cannot be an obstruction to global well-posedness, at least in $H^1_0(\mathcal{D})$.  Killip  and his collaborators\cite{ARMAkillip,Killip-18} studied the solitons and scattering, and the initial-value problem with non-vanishing boundary conditions for equation \eqref{1} in $\mathbb{R}^3$, respectively. Recently, Nguyen and Ricaud \cite{Nguyen-Ricaud24,Nguyen-Ricaud25} derived rigorously the cubic-quintic nonlinear Schr\"odinger semiclassical theory as the mean-field limit of the model and they investigated the behavior of the system in a double-limit. Martel \cite{IMMartel} studied the asymptotic stability of small standing solitary waves of equation \eqref{1} with a defocusing quintic nonlinearity in $\mathbb{R}$, similar conclusions are presented in \cite{PMPMartel,CazenaveCMP,Ohta,CPAMwein} and the references therein. The uniqueness and non-degeneracy of the positive solution for time-independent equation \eqref{1} was shown in \cite{RMPar,Lewin-20}.  Soave \cite{Soave-jde,Soave-jfa} studied the normalized ground states for the nonlinear Schr\"odinger equation with combined nonlinearities (including cubic-quintic nonlinearity). 

A standing wave solution of \eqref{1} is of the form $\Psi(t,x)=e^{-i\mu t}\varPhi(x)$,
where $\mu\in\mathbb{R}$ and $\varPhi(x)\in H^1_0(\mathcal{D})$ is a time-independent function. Then, $\varPhi(x)$ satisfies the following equation
\begin{gather*}
    \begin{aligned}
	-\Delta \varPhi-|\varPhi|^2\varPhi+|\varPhi|^4\varPhi=\mu \varPhi\quad\text{in }\mathcal{D}.
    \end{aligned}
\end{gather*}
Usually, $\varPhi(x)$ is called a \emph{normalized ground state} if it is a minimizer of following minimization problem 
\begin{gather}\label{minN}
    \begin{aligned}
	e(\mathcal{D},N):=\inf\left\lbrace E(\varPhi): \varPhi\in H^1_0(\mathcal{D}), \int_{\mathcal{D}}|\varPhi|^2\,dx=N\right\rbrace,
    \end{aligned}
\end{gather}
where the energy functional $E$ is given by
\begin{gather*}
    \begin{aligned}
        E(\varPhi)=\frac{1}{2}\int_{\mathcal{D}}|\nabla  \varPhi|^2\,dx-\frac{1}{4}\int_{ \mathcal{D}}|\varPhi|^4\,dx+\frac{1}{6}\int_{\mathcal{D}}|\varPhi|^6\,dx.
    \end{aligned}
\end{gather*}

Now, we state the following existence result of minimization problem \eqref{minN}.

\begin{thm}\label{existence-boun}
Let $\mathcal{D} \subset \mathbb{R}^{d}$ ($d=1,2,3$) be a  bounded domain. For any $N>0$, there exists at least one non-negative ground state for $e(\mathcal{D},N)$. Moreover, if $\mathcal{D}$ is a sphere with the origin as its center, then there exists at least one non-negative radially symmetric and decreasing ground state for $e(\mathcal{D},N)$.
\end{thm}

We are interested in the \emph{thermodynamic limit}, in which the size of the domain tends to infinity with fixed density $\rho:= N/|\mathcal{D}|$ and the
normalization changes accordingly. This limit shows the thermodynamic energy per unit volume at density $\rho$, it is crucial because it bridges the microscopic properties of individual particles with the macroscopic thermodynamic properties observed in bulk matter, for more detail about this, one can see \cite{Lieb-book-stability,Correggi-APDE-17} and the references therein.

\begin{thm}[Thermodynamic limit]\label{Thermodynamic-limit}
Let $\mathcal{D} \subset \mathbb{R}^{d}$  ($d=1,2,3$) be a bounded domain, $L>0$ be a constant and $0<\rho\leq1$ be a fixed parameter. Then, for  $L$ is sufficiently large, we have 
\begin{align}\label{TDL}
  \frac{e(\mathcal{D}_L,\rho L^d|\mathcal{D}|)}{L^{d}|\mathcal{D}|}=-\left(\frac{1}{2}-\frac{1}{3m}\right)\frac{1}{2m}+O(L^{-(d+1)})
\end{align}
where  $\mathcal{D}_L=\{Lx:x\in\mathcal{D}\}$ and the constant $m =\min\{\rho^{-1}, 4/3\}$.

Further more, if $0<\rho<1$ and $\mathcal{D}=B(0,R)$ with $R>0$, and let $\varPhi_{L}(x)$ be a non-negative radially symmetric ground state of $e(\mathcal{D}_L,\rho L^d|\mathcal{D}|)$, then there exists a subsequence $\{L_k\}$ such that $L_k\to+\infty$ and
\begin{align}\label{13}
\varPhi_{L_k}\left(\sqrt[d]{\rho|\mathcal{D}|}L_k{x}\right)  \to  u^{TF}(x)=\sqrt{\frac{\mathds{1}_{\Omega}}{m}}(x) \text{ strongly in } L^2\cap L^6(\mathcal{D}_0), \text{ as } k\to+\infty,
\end{align}
where $\Omega=B(0,\sqrt[d]{md/\omega_d})$, $\mathcal{D}_0=B(0,\sqrt[d]{d/(\rho\omega_d)})$, $\mathds{1}_\Omega$ denotes the characteristic function for $\Omega$, and $\omega_d$ denotes the surface area of the unit sphere in $\mathbb{R}^d$. Here $u^{TF}$ is a minimizer of the the following minimization problem
\begin{gather}\label{minTF}
    \begin{aligned}
	e^{TF}(\mathcal{D}_0,1)=\inf\left\lbrace {E}^{TF}(u):  u\in L^6(\mathcal{D}_0), \int_{\mathcal{D}_0}|u|^2\,dx=1\right\rbrace,
    \end{aligned}
\end{gather}
and the Thomas-Fermi energy functional $E^{TF}$ is given by
\begin{gather}\label{fTF}
    \begin{aligned}
	E^{TF}(u)=-\frac{1}{4}\int_{\mathcal{D}_0}|u|^{4}\,dx+\frac{1}{6}\int_{\mathcal{D}_0}|u|^6\,dx.
    \end{aligned}
\end{gather}
\end{thm}
\begin{re}
From \eqref{TDL}, one can find that the thermodynamic limit 
\begin{gather*}
\lim_{L\to+\infty}\frac{e(\mathcal{D}_L,\rho L^d|\mathcal{D}|)}{L^{d}|\mathcal{D}|}=
    \begin{cases}
        -\frac{3}{32},&\text{for } 0<\rho\leq \frac{3}{4},\\
        -\left(\frac{1}{2}-\frac{\rho}{3}\right)\frac{\rho}{2},\quad&\text{for }\frac{3}{4}< \rho\leq 1.
    \end{cases}
\end{gather*}
\end{re}
\begin{re}
Note that $|\mathcal{D}_0|=\rho^{-1}$. For $|\mathcal{D}_0|\geq 4/3$, by the same argument as in \cite{ARMAGontier}, one can obtain the existence of ground states of Thomas-Fermi energy $e^{TF}(\mathcal{D}_0,1)$. For $1\leq|\mathcal{D}_0|<4/3$, we need an interpolation inequality depending on the measure of $\mathcal{D}_0$ to prove this, see Lemma \ref{A}.  For $|\mathcal{D}_0|<1$, i.e. $\rho>1$, we have no idea whether the problem \eqref{minTF} admits a minimizer or not, this also lead to the thermodynamic limit is hard to obtain. 
\end{re}

\begin{re}
    The spherical region $\mathcal{D}=B(0,R)$ means that  the ground state $\varPhi_L(x)$ can be chosen radial symmetry,  then  we can obtain the strongly convergence, by using the  Helly’s selection theorem. 
\end{re}

Next, we consider the $L^{\infty}$-convergence rate of ground states as $L\to+\infty$. The motivation for this part comes from the work by Bethuel,  Brezis and H\'elein \cite{CV1993Asymptotics}, where they considered the Ginzburg-Landau energy
\begin{gather}\label{GLF}
	\begin{aligned}
		E^{GL}_\varepsilon(u)=\frac{1}{2}\int_{\Omega}|\nabla u|^2\,dx
		+\frac{1}{4\varepsilon^2}\int_{\Omega}\left(|u|^2-1\right)^2\,dx,   \ \ \Omega\subset \mathbb{R}^2,
	\end{aligned}
\end{gather}
minimized in $\mathscr{H}_g=\{u\in H^1(\Omega):  u=g\text{ on }\partial \Omega\}$, where $\Omega$ is a smooth bounded simply connected domain and $g:\partial \Omega\to\mathbb{C}$ is a prescribed 
smooth map with $|g|=1$ on $\partial\Omega$ and $deg(g,\partial\Omega)=0$.  In \cite{CV1993Asymptotics}, they proved that $L^\infty$-convergence rate that $\|u_\varepsilon-u_0\|_{L^{\infty}(\Omega)}\leq{O}(\varepsilon^2)$ as $\varepsilon\to0^+$, where $u_\varepsilon$ is a minimizer for $E^{GL}_\varepsilon$ on $\mathscr{H}_g$ and $u_0$ is identified.

Now, we have the following result.
 \begin{thm}\label{convergence}
Assume $0<\rho<3/4$, $\mathcal{D}=B(0,R) \subset \mathbb{R}^{d}$  $(R>0$ and $d=1,2,3)$ and $L>0$. As Theorem \ref{Thermodynamic-limit}, let  $\varPhi_{L}(x)$ be a non-negative radially symmetric ground state of $e(\mathcal{D}_L,\rho L^d|\mathcal{D}|)$ with $\mathcal{D}_L=B(0,LR)$, and $\Omega=B(0,\sqrt[d]{4d/(3\omega_d)})$, $\mathcal{D}_0=B(0,\sqrt[d]{d/(\rho\omega_d)})$. Then, for  $L$ is sufficiently large, we have 
 \begin{align}\label{112}
    \left\| \Delta \varPhi_{L}\left(\sqrt[d]{\rho|\mathcal{D}|}L{x}\right)\right\|_{L^\infty\left(\Omega\right) }\leq  {O}\left(L^{-\frac{2}{3}}\right).
\end{align}
Further more, there exists a subsequence $\{L_k\}$ such that, for $L_k$ is sufficiently large, 
 \begin{align}\label{113}
    \left\|\varPhi_{L_k}\left(\sqrt[d]{\rho|\mathcal{D}|}L_k{x}\right)-u^{TF}(x)\right\|_{L^\infty\left(\Omega\right) }\leq  {O}\left(L_k^{-\frac{2}{3}}\right)
\end{align}
and
\begin{align}\label{114}
   \varPhi_{L_k}\left(\sqrt[d]{\rho|\mathcal{D}|}L_k{x}\right)\leq  {O}\left(  |x|^{-\frac{d-1}{2}}e^{-\frac{L_k}{4}|x|}\right)\ \text{for all } x\in\mathcal{D}_0\setminus\Bar{\Omega}.
\end{align}
\end{thm}
\noindent
\textbf{Comments to Theorem \ref{convergence}:}

1. We should mention that the result about $L^\infty$-convergence rate  in \cite{CV1993Asymptotics}  for  Ginzburg-Landau energy \eqref{GLF} are quite depending on  the specific algebraic properties of the nonlinearity $\int_{\Omega}\left(|u|^2-1\right)^2dx$, their method seem not work for other nonlinearity.   The main contribution  of our result is that we establish a novel  method (including some iterative techniques, uniform energy estimates and  gradient estimates) to prove the $L^\infty$-convergence rate for the cubic-quantic Schr\"{o}dinger model, which may be applicable to other general nonlinearities. 

2. The key point in the proof of  \eqref{113} is to obtain the convergence rate between two Lagrange multipliers
$ \left| \mu_L-\mu^{TF}\right| \leq  O(L^{-1})$ (see Lemma \ref{lem36}). The main difficulty  comes from the term  $\int_{\partial\mathcal{D}_0}|x||\nabla (\varPhi_{L}(\sqrt[d]{\rho|\mathcal{D}|}L{x}))|^2dS$, which will appear in the Pohozaev identity.  For $0<\rho<3/4$, utilizing  some elliptic estimates carefully we can obtain a good gradient estimate of $\varPhi_{L}(\sqrt[d]{\rho|\mathcal{D}|}L{x})$ on $\partial\mathcal{D}_0$  as $L\to+\infty$, which is exponential decay with respect to $L$.
However, for $3/4\leq\rho\leq1$, the corresponding gradient estimate of $\varPhi_{L}(\sqrt[d]{\rho|\mathcal{D}|}L{x})$ on $\partial\mathcal{D}_0$ is bad,  which lead to the  $L^\infty$-convergence rate like \eqref{113} is hard to obtain, we leave it as an open problem. 
\vspace{3mm}

The thermodynamic limit is a large-volume limit, finally we consider another limit in the whole space $\mathbb{R}^d$ with the number of particles $N\rightarrow{+\infty}$. This limit can be  viewed as a Thomas-Fermi limit, one can see 
the work by Lieb, Seiringer and Yngvason \cite{PRALieb,CMPLieb}, where they considered the following   minimization problem 
\begin{gather*}
	\begin{aligned} 
	e_V^{GP}(N,a)=\left\lbrace E_{aV}^{GP}(\varPhi):\varPhi\in H^1(\mathbb{R}^3): \int_{\mathbb{R}^3}|\varPhi|^2\,dx=N\right\rbrace,
	\end{aligned}
\end{gather*}
where $E_{aV}^{GP}$ denotes the Gross-Pitaevskii energy functional  with trapping potential 
\begin{gather*}
	\begin{aligned} 
		E_{aV}^{GP}(\varPhi)=\int_{\mathbb{R}^3}\left(|\nabla \varPhi|^2+V(x)|\varPhi|^2\right)\,dx
		+4\pi a\int_{\mathbb{R}^3}|\varPhi|^4\,dx,
	\end{aligned}
\end{gather*}
and $a>0$ denotes the scattering length. 
They showed that up to a scaling the ground states $\varPhi_N$ of $e_V^{GP}(N,a)$  would  converge to a Thomas-Fermi minimizer as $N\rightarrow +\infty$.  
Here we consider the Thomas-Fermi limit of  ground states $\varPhi_{N}$ of $e(\mathbb{R}^d,N)$ (see \eqref{minN}) as $N\to+\infty$.

 \begin{thm}[Thomas-Fermi limit]\label{th14}
 Assume  that $d=1,2,3$, and $\varPhi_{N}(x)$ is a non-negative radially symmetric ground state of $e(\mathbb{R}^d,N)$. Then, for $N$ is sufficiently large, we have 
\begin{align}\label{116}
\frac{e(\mathbb{R}^d,N)}{N}=-\frac{3}{32}+O(N^{-1/d}).
\end{align}
Further more, there exists a subsequence $\{N_k\}$ such that, for $N_k$ is sufficiently large,
\begin{align}\label{115}
\varPhi_{N_k}\left(N_k^{\frac{1}{d}}x\right)  \to  u^{TF}(x)=\sqrt{\frac{3\cdot\mathds{1}_{\Omega}}{4}}(x) \ \ \text{strongly in}\  L^2\cap L^6(\mathbb{R}^d),
\end{align}
\begin{align}\label{10}
\left\|\varPhi_{N_k}\left(N_k^{\frac{1}{d}}x\right)-u^{TF}(x)\right\|_{L^\infty\left(\Omega\right) }\leq  {O}\left( N_k^{-\frac{2}{3d}}\right),
\end{align}
and
\begin{align}\label{11}
    \varPhi_{N_k}\left(N_k^{\frac{1}{d}}x\right)\leq  {O}\left(  |x|^{-\frac{d-1}{2}}e^{-\frac{N_k^{1/d}}{4}|x|}\right) \quad\text{for all }x\in\mathbb{R}^d\setminus\Bar{\Omega},
\end{align}
where $\Omega=B\big(0,\sqrt[d]{4d/(3\omega_d)}\big)$ and $u^{TF}$ is a minimizer of \eqref{minTF} for $\mathcal{D}_0=\mathbb{R}^d$. 
\end{thm}
We should note that by similar argument as the thermodynamic limit (see Theorem \ref{Thermodynamic-limit}), one can obtain this result. In addition, for the result about Thomas-Fermi limit of cubic-quintic Schr\"odinger model with radial potential, see our work \cite{lw-arxiv}.
 \vspace{3mm}
	
This paper is organized as follows. In section \ref{sec2}, we give some basic results and prove Theorem \ref{existence-boun}. In section \ref{sec3}, we consider the thermodynamic limit of the cubic-quintic Schr\"{o}dinger energy, and study the asymptotic behavior of ground states and prove Theorem \ref{Thermodynamic-limit} and \ref{convergence}. In Section \ref{sec4}, we study the Thomas-Fermi limit for the model in the whole space, and prove Theorem \ref{th14}. The final section is Appendix. 

Throughout this paper, we use standard notations. For simplicity, we write $\|\cdot\|_p$ denotes the $L^p(\mathcal{D})$ norm for $p\geq 1$; $\rightharpoonup$ denotes weakly converge; $\mathds{1}_\Omega(x)$ denotes the characteristic function of the set $\Omega\subset\mathbb{R}^3$;  $X\sim Y$ denotes $X\lesssim Y$ and  $Y\lesssim X$, where $X\lesssim Y\ (X\gtrsim Y)$ denotes $X\leq CY\ (X\geq CY)$ for some appropriate positive constants $C$. The value of positive constant $C$ is allowed to change from line to line and also in the same formula.

\section{Preliminaries}\label{sec2}
In this section, we give some basic results and aim to prove Theorem \ref{existence-boun}.
For $\mathcal{D}=\mathbb{R}^d$, the existence and uniqueness of ground state $\varPhi_N$ of $e(\mathbb{R}^d,N)$ (see \eqref{minN}) has already been considered in \cite{HB-83,HB-83A,JS-00,ARMAkillip,JJ-10,RMPar}. We summarize all of these results in the proposition below.
\begin{pro} \label{pro1}
Assume that $\mathcal{D}=\mathbb{R}^d$  ($d=1,2,3$)  and  $N>0$. Then, there exists a critical mass $N_*(d)$ such that

{\rm(i)} if $d=1$, there exists at least one non-negative minimizer of $e(\mathbb{R}^d,N)$ for $N>N_*(1)=0$.
		
{\rm(ii)} if $d=2$, there is no minimizer of $e(\mathbb{R}^d,N)$ for $0<N\leq N_*(2)$; and there exists at least one non-negative minimizer of $e(\mathbb{R}^d,N)$ for any $N>N_*(2)$. Here $N_*(2)$ corresponds to the best constant of the following  Gagliardo-Nirenberg  inequality
\begin{gather}\label{gn}
    \begin{aligned}
	\int_{\mathbb{R}^2}|u|^4\,dx\leq\frac{2}{N_*(2)}\int_{\mathbb{R}^2}|\nabla u|^2\,dx\int_{\mathbb{R}^2}|u|^2\,dx.
    \end{aligned}
\end{gather}
		
{\rm(iii)} if $d=3$, there is no minimizer of $e(\mathbb{R}^d,N)$ for $0<N<N_*(3)$; and there exists at least one non-negative minimizer of $e(\mathbb{R}^d,N)$ for any $N\geq N_*(3)$. Here $N_*(3)$  corresponds to the best constant of the Gagliardo-Nirenberg-H\"older inequality
\begin{gather}\label{gnh}
    \begin{aligned}
	\int_{\mathbb{R}^3}|u|^4\,dx\leq \frac{8}{3\sqrt{N_*(3)}}\left(\int_{\mathbb{R}^3}|u|^2\,dx\right)^{\frac{1}{2}}\left(\int_{\mathbb{R}^3}|u|^6\,dx\right)^{\frac{1}{4}}\left(\int_{\mathbb{R}^3}|\nabla u|^2\,dx\right)^{\frac{3}{2}}.
    \end{aligned}
\end{gather}
\noindent Moreover, the non-negative minimizer of  $e(\mathbb{R}^d,N)$ is a $C^2(\mathbb{R}^d)$ function, unique, radially symmetric and decreasing with respect to $|x|$.

\end{pro}

\begin{re}
For the proofs of inequalities \eqref{gn} and \eqref{gnh} above, one can see \cite{CMPwein,ARMAkillip}.
\end{re}
Recall in \cite{CMPwein} the well-known Gagliardo-Nirenberg inequality with the best constant:
\begin{gather*}
    \begin{aligned}
	\int_{\mathbb{R}^d}|u|^4\,dx\leq\frac{2}{\|Q_d\|_2^{2}}\left( \int_{\mathbb{R}^d}|\nabla  u|^2\,dx\right)^{\frac{d}{2}}\left(\int_{\mathbb{R}^d}|u|^2\,dx\right)^{\frac{4-d}{2}}
    \end{aligned}
\end{gather*}
with equality only for $u=Q_d$, where, up to translations, $Q_d$ is the unique positive radial ground state solution to the following equation
\begin{gather}\label{Qeq}
    \begin{aligned}
    -\frac{d}{2}\Delta u+\frac{4-d}{2}u=|u|^{2}u,\quad u\in H^1(\mathbb{R}^d).
    \end{aligned}
\end{gather}
and $Q_d$ satisfies
\begin{gather}\label{Qid}
    \begin{aligned}
	\int_{\mathbb{R}^d}|\nabla  Q_d|^2\,dx=\int_{\mathbb{R}^d}|Q_d|^2\,dx=\frac{1}{2}\int_{\mathbb{R}^d}|Q_d|^4\,dx.
    \end{aligned}
\end{gather}

In the following, we will give the proof of Theorem \ref{existence-boun}.

\begin{proof}[\textbf{Proof of Theorem \ref{existence-boun}:}]
Recall from \cite{anailsis}, we have $|\nabla |\varPhi||\leq|\nabla \varPhi|$, a.e., $x\in\mathbb{R}^d$. Thus, we deduce that $E_\mathcal{D}(\varPhi)\geq E_\mathcal{D}(|\varPhi|)$. Therefore, it suffices to consider the minimization problem \eqref{minN} in real non-negative function space.

For any $N > 0$, let $\{\varPhi_k\}_{k=1}^\infty\subset H^1_0(\mathcal{D})$ with $ \|\varPhi_k\|_{2}^{2} =N$ be a minimizing sequence of $e(\mathcal{D},N)$. By H\"older's inequality and Young's inequality, we have
\begin{gather*}
\begin{aligned}
	\int_{\mathcal{D}}|\varPhi_k|^4\,dx\leq\left( \int_{\mathcal{D}}|\varPhi_k|^2\,dx\right)^{\frac{1}{2}}\left( \int_{\mathcal{D}}|\varPhi_k|^6\,dx\right)^{\frac{1}{2}}
    \leq\frac{2}{3}\int_{\mathcal{D}}|\varPhi_k|^6\,dx+\frac{3N}{8}.
\end{aligned}
\end{gather*}
Hence, 
\begin{gather*}
\begin{aligned}
	E(\varPhi_k)&=\frac{1}{2}\int_{\mathcal{D}}|\nabla \varPhi_k|^2\,dx-\frac{1}{4}\int_{\mathcal{D}}|\varPhi_k|^4\,dx+\frac{1}{6}\int_{\mathcal{D}}|\varPhi_k|^6\,dx\\
	&\geq\frac{1}{2}\int_{\mathcal{D}}|\nabla \varPhi_k|^2\,dx-\frac{3N}{32}.
\end{aligned}
\end{gather*}
This implies that the minimizing sequence $\{\varPhi_k\}_{k=1}^\infty$ is uniformly bounded in $H^1_0(\mathcal{D})$ for any $N>0$. Therefore, up to a subsequence, $\varPhi_k\rightharpoonup \varPhi_N$ weakly in $H^1_0(\mathcal{D})$. By Sobolev embedding, we obtain
\begin{gather*}
\begin{aligned}
	\varPhi_k\to \varPhi_N\ \text{ strongly in } L^p(\mathcal{D})\ \text{ for }2\leq p<2^*,
\end{aligned}
\end{gather*}
where $2^*=+\infty$ if $d=1,2$ and $2^*=6$ if $d=3$. Thus, we conclude that $\|\varPhi_N\|^2_{2}=N$ and $E(\varPhi_N)=e(\mathcal{D},N)$, by the weakly lower semi-continuity. This implies that $\varPhi_N$ is a minimizer of $e(\mathcal{D},N)$.

Finally, we show that there exists at least one non-negative ground state  $\varPhi_{N}(x)$ is radially symmetric and decreasing with respect to $|x|$ for $\mathcal{D}=B(0,R)$. Let $\varPhi_{N}^*(x)$ be the symmetric-decreasing rearrangement of $\varPhi_{N}(x)$ and $\mathcal{D}^*$ be the symmetric rearrangement of $\mathcal{D}$, i.e., the symmetric rearrangement $\mathcal{D}^*$ is the open centered ball whose volume agrees with $\mathcal{D}$,
\begin{gather}\label{def}
\begin{aligned}
	\mathcal{D}^*=\{x\in\mathbb{R}^d\  | \ \omega_d|x|^d< {\rm Vol}(\mathcal{D})\}.
\end{aligned}
\end{gather}	
Thus, from \cite{anailsis} that for any $s\geq1$,
\begin{gather*}
\begin{aligned}
	\int_{\mathcal{D}^*}|\varPhi_{N}^*(x)|^s\,dx
	=\int_{\mathcal{D}}|\varPhi_{N}(x)|^s\,dx\ \text{ and}\ \int_{\mathcal{D}^*}|\nabla \varPhi_{N}^*(x)|^2\,dx
	\leq\int_{\mathcal{D}}|\nabla \varPhi_{N}(x)|^2\,dx.
\end{aligned}
\end{gather*}	
Using \eqref{def}, we then have $\mathcal{D}^*=\mathcal{D}$ for  $\mathcal{D}=B(0,R)$. It follows that
\begin{gather*}
\begin{aligned}
	\int_{\mathcal{D}}|\varPhi_{N}^*(x)|^s\,dx
	=\int_{\mathcal{D}}|\varPhi_{N}(x)|^s\,dx\ \text{ and}\ \int_{\mathcal{D}}|\nabla \varPhi_{N}^*(x)|^2\,dx
	\leq\int_{\mathcal{D}}|\nabla \varPhi_{N}(x)|^2\,dx.
\end{aligned}
\end{gather*}
This implies that
\begin{gather*}
\begin{aligned}
	e(\mathcal{D},N)\leq E(\varPhi_{N}^*)\leq E(\varPhi_{N})=e(\mathcal{D},N). 	
\end{aligned}
\end{gather*}
Thus, $\varPhi_{N}^*(x)$ is a minimizer of $e(\mathcal{D},N)$ and we complete the proof of Theorem \ref{existence-boun}.
\end{proof}

\section{Thermodynamic limit}\label{sec3}
In this section, we prove Theorem \ref{Thermodynamic-limit} and \ref{convergence}. Assume that $\varPhi_{L}(x)$ is a non-negative minimizer  of $e(\mathcal{D}_L,\rho L^d|\mathcal{D}|)$   with $\mathcal{D}_L=\{Lx:x\in\mathcal{D}\}$ and $\mathcal{D} \subset \mathbb{R}^{d}$ is a bounded domain. And $\varPhi_{L}(x)$ satisfies the following Euler-Lagrange equation
\begin{gather*}
	\begin{aligned}
	-\Delta \varPhi_L-\varPhi_L^3+\varPhi_L^5=\mu_L \varPhi_L,\quad x\in\mathcal{D}_L,
	\end{aligned}
\end{gather*}
where $\mu_L\in\mathbb{R}$ is a Lagrange multiplier associated to $\varPhi_L$.

\subsection{Energy estimates on the general bounded domain}	

\begin{lem}\label{TLer}
Let $\mathcal{D} \subset \mathbb{R}^{d}$ ($d=1,2,3$) be a bounded domain, $L>0$ be a constant  and $0<\rho\leq1$ be a fixed parameter. Then, for  $L$ is sufficiently large,  we have 
\begin{align*}
-\left(\frac{1}{2}-\frac{1}{3m}\right)\frac{\rho L^d|\mathcal{D}|}{2m}\leq e(\mathcal{D}_L,\rho L^d|\mathcal{D}|)\leq-\left(\frac{1}{2}-\frac{1}{3m}\right)\frac{\rho L^d|\mathcal{D}|}{2m}+{O}(L^{-1}),
\end{align*}
where $m=\min\{\rho^{-1},4/3\}$.
\end{lem}
\begin{pf} The proof of the upper bound is divided into the following two cases.

Case 1: $0<\rho\leq3/4$. 
For any $\varphi\in H^1_0(\mathcal{D}_L)$, using the H\"older's inequality and Young's inequality, we get
\begin{gather*}
    \begin{aligned}
    \int_{\mathcal{D}_L}|\varphi|^4\,dx&\leq\left(\int_{\mathcal{D}_L}|\varphi|^2\,dx\right)^{1/2}\left(\int_{\mathcal{D}_L}|\varphi|^6\,dx\right)^{1/2}
    \leq\frac{3}{8}\int_{\mathcal{D}_L}|\varphi|^2\,dx+\frac{2}{3}\int_{\mathcal{D}_L}|\varphi|^6\,dx.
    \end{aligned}
\end{gather*}
Hence, for any $v\in H^1_0(\mathcal{D}_L)$ with $\|v\|_{L^2(\mathcal{D}_L)}^2=\rho L^d|\mathcal{D}|$, we obtain
\begin{gather*}
    \begin{aligned}
    {E}(v)\geq-\frac{1}{4}\int_{\mathcal{D}_L}|v|^4\,dx+\frac{1}{6}\int_{\mathcal{D}_L}|v|^6\,dx\geq-\frac{3\rho L^d|\mathcal{D}|}{32}.
    \end{aligned}
\end{gather*}
Thus, for any $L>0$, we have
\begin{align}\label{Ts1}
e(\mathcal{D}_L,\rho L^d|\mathcal{D}|)\geq-\frac{3\rho L^d|\mathcal{D}|}{32}=-\left(\frac{1}{2}-\frac{1}{3m}\right)\frac{\rho L^d|\mathcal{D}|}{2m},
\end{align}
where $m=3/4$ for $0<\rho\leq3/4$.

Case 2: $3/4<\rho\leq1$. For any $\varphi\in H^1_0(\mathcal{D}_L)$, using the H\"older's inequality, we obtain
\begin{gather*}
    \begin{aligned}
   \int_{\mathcal{D}_L}|\varphi|^4\,dx\leq\left(\int_{\mathcal{D}_L}\,dx\right)^{\frac{4\rho-3}{3}}\left(\int_{\mathcal{D}_L}|\varphi|^2\,dx\right)^{2(1-\rho)}\left(\int_{\mathcal{D}_L}|\varphi|^6\,dx\right)^{\frac{2\rho}{3}}.
    \end{aligned}
\end{gather*}
From $3/4<\rho\leq1$ and Young's inequality, we get
\begin{gather*}
    \begin{aligned}
    \int_{\mathcal{D}_L}|\varphi|^4\,dx&\leq \left(L^d|\mathcal{D}|\right)^{\frac{4\rho-3}{3}}\left(\int_{\mathcal{D}_L}|\varphi|^2\,dx\right)^{2(1-\rho)}\left(\int_{\mathcal{D}_L}|\varphi|^6\,dx\right)^{\frac{2\rho}{3}}\\
    &\leq\frac{3-2\rho}{3}\rho^{\frac{2\rho}{3-2\rho}}\left(L^d|\mathcal{D}|\right)^{\frac{4\rho-3}{3-2\rho}}\left(\int_{\mathcal{D}_L}|\varphi|^2\,dx\right)^{\frac{6(1-\rho)}{3-2\rho}}+\frac{2}{3}\int_{\mathcal{D}_L}|\varphi|^6\,dx.
    \end{aligned}
\end{gather*}
Hence, for any $v\in H^1_0(\mathcal{D}_L)$ with $\|v\|_{L^2(\mathcal{D}_L)}^2=\rho L^d|\mathcal{D}|$, we obtain
\begin{gather*}
    \begin{aligned}
    E(v)&\geq-\frac{3-2\rho}{12}\rho^{\frac{2\rho}{3-2\rho}}\left(L^d|\mathcal{D}|\right)^{\frac{4\rho-3}{3-2\rho}}\left(\int_{\mathcal{D}_L}|v|^2\,dx\right)^{\frac{6(1-\rho)}{3-2\rho}}\\
    &=-\frac{3-2\rho}{12}\rho^2L^d|\mathcal{D}|=-\left(\frac{1}{2}-\frac{1}{3m}\right)\frac{\rho L^d|\mathcal{D}|}{2m},
    \end{aligned}
\end{gather*}
where $m=\rho^{-1}$ for $3/4<\rho\leq1$.

To prove the lower bound, let $\varPhi_L(x)$ be a non-negative minimizer of $e(\mathcal{D}_L,\rho L^d|\mathcal{D}|)$ for any fixed $L>0$. Now, we define
\begin{align}\label{scaling}
    w_L(x)=\varPhi_L(\sqrt[d]{\rho|\mathcal{D}|}L{x}),\quad x\in\mathcal{D}_0=\big\{x\in\mathbb{R}^d:\sqrt[d]{\rho |\mathcal{D}|}x\in \mathcal{D}\big\},
\end{align}
such that $\|w_L\|_{L^2(\mathcal{D}_0)}^2=1$. 
Direct calculation yields $|\mathcal{D}_0|=\rho^{-1}$ and
\begin{gather}\label{T211}
    \begin{aligned}
    e(\mathcal{D}_L,\rho L^d|\mathcal{D}|)=E(\varPhi_L)=\rho L^d|\mathcal{D}|E_L(w_L),
    \end{aligned}
\end{gather}
where 
\begin{align*}
   E_L(w_L)=\frac{L^{-2}}{2(\rho|\mathcal{D}|)^{\frac{2}{d}}}\int_{ \mathcal{D}_0}|\nabla w_L|^2\,dx-\frac{1}{4}\int_{\mathcal{D}_0}|w_L|^{4}\,dx+\frac{1}{6}\int_{ \mathcal{D}_0}|w_L|^6\,dx.
\end{align*}
Thus, $w_L$ is a minimizer, i.e.,
		\begin{gather*}
			\begin{aligned}
		E_L(w_L)=\inf\bigg\{E_L(w):w\in H^1_0(\mathcal{D}_0),\ \int_{\mathcal{D}_0}|w|^{2}\,dx=1\bigg\}.
			\end{aligned}
		\end{gather*}
Let $\eta(x)\in C^{\infty}_0(\mathbb{R}^d)$ be a non-negative smooth cut-off function defined by
\begin{gather*}
	\begin{cases}
		\eta_L(x)=1
		& \ \   \text{for }\ x\in\{x\in\mathcal{D}_0 :dist(x,\partial\mathcal{D}_0)\geq2f(L)\},\\
		0\leq\eta_L(x)\leq1
		& \ \   \text{for }\  x\in\{x\in\mathcal{D}_0 :f(L)\leq dist(x,\partial\mathcal{D}_0)\leq2f(L)\},\\
		\eta_L(x)=0
		&  \ \   \text{for other}.
	\end{cases}
\end{gather*}
Here $f(L)>0$ is to be determined later and satisfies $f(L)\to0$ as $L\to+\infty$. Without loss of generality, we may assume that $|\nabla\eta_L(x)|\leq{M}/{f(L)}$, where $M > 0$ is a positive constant independent of $L$. From Lemma \ref{A} and $|\mathcal{D}_0|=\rho^{-1}\geq1$, we can choose
the support set $\Omega$ of $u^{TF}=\sqrt{{\mathds{1}_{\Omega}}/{m}}$ such that $\Omega\subseteq\mathcal{D}_0$ and $m=\min\{\rho^{-1},4/3\}$. Consider the function
\begin{gather}\label{5211}
	\begin{aligned}
		u^{TF}_L(x):=A_L\eta_L(x)u^{TF}(x),\ \ \ x\in\mathcal{D}_0,
	\end{aligned}
\end{gather}
where $A_L$ is chosen such that $\int_{\mathcal{D}_0}|u^{TF}_L|^2\,dx=1$. Define
\begin{gather*}
    \begin{aligned}
	\Omega^{(1)}_L&:=\left\lbrace x\in\mathcal{D}_0:dist(x,\partial\mathcal{D}_0)\leq f(L) \right\rbrace;\\
	\Omega^{(2)}_L&:=\left\lbrace x\in\mathcal{D}_0:f(L)\leq dist(x,\partial\mathcal{D}_0)\leq2f(L) \right\rbrace.
    \end{aligned}
\end{gather*}
Obviously, $u^{TF}_L\in C^{\infty}_0(\mathcal{D}_0)$ and $u^{TF}_L(x)\to u^{TF}(x)$ strongly in $L^r(\mathcal{D}_0)$, where $2\leq r\leq 2^*$. Thus, we deduce that
\begin{gather*}
	\begin{aligned}
		E^{TF}(u^{TF}_L)=E^{TF}(u^{TF})+o(1),
	\end{aligned}
\end{gather*}
where $E^{TF}$ be given by \eqref{fTF}.

On the one hand, using $\int_{\mathcal{D}_0}|u^{TF}_L|^2\,dx=1$, we obtain
\begin{gather*}
	\begin{aligned}
		\frac{1}{A_L^2}&=	\int_{\{x\in\mathcal{D}_0 :dist(x,\partial\mathcal{D}_0)\geq f(L)\}}\eta^2_L(x)|u^{TF}|^2\,dx\\
		&=\int_{\mathcal{D}_0}|u^{TF}|^2\,dx+\int_{\Omega^{(2)}_L}\left( \eta^2_L(x)-1\right) |u^{TF}|^2\,dx-\int_{\Omega^{(1)}_L}|u^{TF}|^2\,dx\\
		&\geq1-\int_{\Omega^{(1)}_L\cup\Omega^{(2)}_L}|u^{TF}|^2\,dx\\
		&\geq1-{O}(f(L))\ \ \ \text{as }L\to+\infty.
	\end{aligned}
\end{gather*}
This means that
\begin{gather}\label{5241}
	\begin{aligned}
		1\leq{A_L^2}\leq1+{O}(f(L))\ \ \ \text{as }L\to+\infty.
	\end{aligned}
\end{gather}
By direct computations, we derive from \eqref{5211} and \eqref{5241} that
\begin{gather}\label{hnable31}
	\begin{aligned}
\int_{\mathcal{D}_0}|\nabla u^{TF}_L|^2\,dx
&=A_L^2\int_{\mathcal{D}_0}\left( |\nabla\eta_L u^{TF}|^2+2\nabla\eta_L u^{TF}\eta_L\nabla u^{TF}+|\eta_L\nabla u^{TF}|^2\right) \,dx\\
&=A_L^2\int_{\Omega^{(1)}_L\cup\Omega^{(2)}_L}|\nabla\eta_L u^{TF}|^2\,dx\leq{O}\left(\frac{1}{f(L)}\right)\ \ \ \text{as }L\to+\infty,
	\end{aligned}
\end{gather}
where in the last step we have used  $|\nabla\eta_L(x)|\leq{M}/{f(L)}$.
On the other hand, we deduce from \eqref{5211} and \eqref{5241} that
\begin{gather}\label{4231}
	\begin{aligned}
		0&\leq E^{TF}(u^{TF}_L)-E^{TF}(u^{TF})\\
		&=\frac{1}{6}\int_{\mathcal{D}_0}\left(|u^{TF}_L|^6-|u^{TF}|^6\right)\,dx+\frac{1}{4}\int_{\mathcal{D}_0}\left(|u^{TF}|^4-|u^{TF}_L|^4\right)\,dx\\
		&\leq C\int_{\Omega}\left|A_L^6\eta_L^6(x)-1\right|\,dx+C\int_{\Omega}\left|A_L^4\eta_L^4(x)-1\right|\,dx\\
		&\leq {O}(f(L))\ \ \ \text{as }L\to+\infty.
	\end{aligned}
\end{gather}
It follows from $\Omega\subseteq\mathcal{D}_0$, \eqref{scaling}, \eqref{hnable31} and \eqref{4231} that
\begin{gather}\label{similarly1}
    \begin{aligned}
	E^{TF}(u^{TF})
    &=-\frac{1}{4}\int_{\mathcal{D}_0}|u^{TF}|^{4}\,dx+\frac{1}{6}\int_{\mathcal{D}_0}|u^{TF}|^6\,dx\\
    &\leq-\frac{1}{4}\int_{\mathcal{D}_0}|w_L|^{4}\,dx+\frac{1}{6}\int_{\mathcal{D}_0}|w_L|^6\,dx\\
    &\leq E_{L}(w_L)\leq
	E_{L}(u^{TF}_L)\\
 &\leq E^{TF}(u^{TF}_L)+\frac{L^{-2}}{2(\rho|\mathcal{D}|)^{\frac{2}{d}}}\int_{\mathcal{D}_0}|\nabla u^{TF}_L|^2\,dx\\
	&\leq E^{TF}(u^{TF})+{O}(f(L))+{O}\left(\frac{L^{-2}}{f(L)}\right)\ \ \ \text{as }L\to+\infty.
    \end{aligned}
\end{gather}	
 Thus, taking $f(L)=L^{-1}$, we deduce from \eqref{similarly1}, \eqref{T211} and $E^{TF}(u^{TF})=-\left(\frac{1}{2}-\frac{1}{3m}\right)\frac{1}{2m}$ (see Lemma \ref{A}) that
\begin{gather*}
    \begin{aligned}
    e(\mathcal{D}_L,\rho L^d|\mathcal{D}|)\leq-\left(\frac{1}{2}-\frac{1}{3m}\right)\frac{\rho L^d|\mathcal{D}|}{2m}+{O}(L^{-1})\quad\text{as }L\to+\infty.
    \end{aligned}
\end{gather*}	
Hence, combining \eqref{Ts1}, we completes the proof of Lemma \ref{TLer}.
\end{pf}

Next, we give the following sharp estimates about the nonlinear term.
\begin{lem}\label{lem22}
Let $\mathcal{D} \subset \mathbb{R}^{d}$ ($d=1,2,3$) be a bounded domain, $L>0$ be a constant  and $0<\rho\leq1$ be a fixed parameter.  Assume that $\varPhi_N(x)$ is a minimizer of $e(\mathcal{D}_L,\rho L^d|\mathcal{D}|)$, we then have
\begin{gather*}
    \begin{aligned}
        \int_{\mathcal{D}_L}|\varPhi_N|^4\,dx\sim \int_{\mathcal{D}_L}|\varPhi_N|^6\,dx\sim \rho L^d|\mathcal{D}|\qquad\text{as}\ L\to+\infty.
    \end{aligned}
\end{gather*}
\end{lem}
\begin{pf}
It follows from Lemma \ref{TLer} that 
\begin{align*}
    \frac{1}{6}\int_{\mathcal{D}_L}|\varPhi_N|^6\,dx\leq\frac{1}{4}\int_{\mathcal{D}_L}|\varPhi_N|^4\,dx.
\end{align*}
The H\"older's inequality yields
\begin{gather*}
    \begin{aligned}
        \frac{1}{4}\int_{\mathcal{D}_L}|\varPhi_N|^4\,dx\leq\frac{(\rho L^d|\mathcal{D}|)^{1/2}}{4}\left(\int_{\mathcal{D}_L}|\varPhi_N|^6\,dx\right)^{1/2}.
    \end{aligned}
\end{gather*}
From the above two inequalities, we get
\begin{align*}
    \int_{\mathcal{D}_L}|\varPhi_N|^6\,dx\leq \frac{9}{4}\rho L^d|\mathcal{D}|.
\end{align*}
Then, we have
\begin{gather}\label{231}
    \begin{aligned}
\int_{\mathcal{D}_L}|\varPhi_N|^4\,dx\leq (\rho L^d|\mathcal{D}|)^{1/2}\left(\int_{\mathcal{D}_L}|\varPhi_N|^6\,dx\right)^{1/2}\leq\frac{3}{2}\rho L^d|\mathcal{D}|.
    \end{aligned}
\end{gather}
On the other hand, we deduce from Lemma \ref{TLer} that
\begin{gather*}
    \begin{aligned}
        -\frac{1}{4}\int_{\mathcal{D}_L}|\varPhi_N|^4\,dx\leq e(\mathcal{D}_L,\rho L^d|\mathcal{D}|)\lesssim-\rho L^d|\mathcal{D}|.
\end{aligned}
\end{gather*}
Thus, we have
\begin{gather*}
    \begin{aligned}
    \int_{\mathcal{D}_L}|\varPhi_N|^4\,dx\gtrsim \rho L^d|\mathcal{D}|.
\end{aligned}
\end{gather*}
Combining \eqref{231}, we complete the proof of this lemma.
\end{pf}
\subsection{Energy estimates on the spherical domain}
In this subsection, we consider $\mathcal{D}=B(0,R)$, where $R>0$, we then have $\mathcal{D}_L=\{Lx:x\in B(0,R)\}=B(0,LR)$. From Theorem \ref{existence-boun}, it is known that the ground state $\varPhi_L(x)$ of $e(\mathcal{D}_L,\rho L^d|\mathcal{D}|)$ is non-negative radially symmetric and decreasing function. For the sake of simplicity, define 
\begin{align*}
	 w_L(x)=\varPhi_L(\sqrt[d]{\rho|\mathcal{D}|}L{x}),\quad x\in\mathcal{D}_0=B\left(0,\sqrt[d]{d/(\rho\omega_d)}\right).
\end{align*}
 Note that 
\begin{gather*}
    \begin{aligned}
    \int_{\mathcal{D}_0}|w_L(x)|^2\,dx=(\rho L^d|\mathcal{D}|)^{-1}\int_{\mathcal{D}_L}|\varPhi_L({x})|^2\,dx=1.
    \end{aligned}
\end{gather*}
We deduce from Lemma \ref{lem22} that
\begin{gather}\label{29}
    \begin{aligned}
   \int_{\mathcal{D}_0}|w_L(x)|^4\,dx\sim \int_{\mathcal{D}_0}|w_L(x)|^6\,dx\sim1\qquad\text{as}\ N\to+\infty.
    \end{aligned}
\end{gather}
Now, we introduce the following constraint minimization problem
\begin{gather}\label{min1}
    \begin{aligned}
   e_L(\mathcal{D}_0,1):=\inf\left\{{E}_{L}(u): u\in H^{1}_0(\mathcal{D}_0), \int_{\mathcal{D}_0}|u|^2\,dx =1\right\},
    \end{aligned}
\end{gather}
where the energy functional ${E}_{L}$ be given by
\begin{gather*}
    \begin{aligned}
    E_L(u)=\frac{L^{-2}}{2(\rho|\mathcal{D}|)^{\frac{2}{d}}}\int_{ \mathcal{D}_0}|\nabla u|^2\,dx-\frac{1}{4}\int_{\mathcal{D}_0}|u|^{4}\,dx+\frac{1}{6}\int_{ \mathcal{D}_0}|u|^6\,dx.
    \end{aligned}
\end{gather*}
 Thus, $w_L$ is also non-negative radially symmetric ground state of $ e_L(\mathcal{D}_0,1)$, i.e.,
\begin{gather}\label{45}
    \begin{aligned}
	e(\mathcal{D}_L,\rho L^d|\mathcal{D}|)=E(\varPhi_L)=\rho L^d|\mathcal{D}|E_L(w_L)=\rho L^d|\mathcal{D}|e_L(\mathcal{D}_0,1).
    \end{aligned}
\end{gather}
And $w_L$ satisfies the following Euler-Lagrange equation
\begin{gather}\label{hEL}
    \begin{aligned}
    -\frac{L^{-2}}{(\rho|\mathcal{D}|)^{\frac{2}{d}}}\Delta  w_L(x)-  w_L^3(x)+w_L^5(x)=\mu_L w_L(x),\  x\in\mathcal{D}_0,
    \end{aligned}
\end{gather}
where $\mu_L\in\mathbb{R}$ is a Lagrange multiplier associated to $w_L$.	

Now assume that $\Omega=B(0,\sqrt[d]{md/\omega_d})$ such that  $m =\min\{\rho^{-1}, 4/3\}$ and taking
\begin{align}\label{38}
    u^{TF}(x)=\sqrt{\frac{\mathds{1}_{\Omega}}{m}}(x),\quad x\in\mathcal{D}_0=B\left(0,\sqrt[d]{d/(\rho\omega_d)}\right),
\end{align}
is the unique radially symmetric non-increasing minimizer of \eqref{minTF}.
Here $\mathds{1}_\Omega(x)$ denotes the characteristic function for a Borel set $\Omega\subseteq\mathcal{D}_0$ and $\omega_d$ denotes the surface area of the unit sphere in $\mathbb{R}^d$. In fact, note that $m=\min\{\rho^{-1}, 4/3\}$ and $|\mathcal{D}_0|=\rho^{-1}\geq1$, then, from Lemma \ref{A}, this assumption holds. Moreover, $u^{TF}$ satisfies the following Euler-Lagrange equation
\begin{gather}\label{357}
    \begin{aligned}
    \big(u^{TF}(x)\big)^5-\big(u^{TF}(x)\big)^3=\mu^{TF} u^{TF}(x),\quad x\in\mathcal{D}_0,
    \end{aligned}
\end{gather}
where $\mu^{TF}\in\mathbb{R}$ is a Lagrange multiplier and its value can be obtained, that is,
\begin{gather*}
    \begin{aligned}
		\mu^{TF}=\big(u^{TF}\big)^4-\big(u^{TF}\big)^2=-\left(1-\frac{1}{m}\right)\frac{1}{m}.
    \end{aligned}
\end{gather*}
\begin{re}
From \eqref{38}, by the direct calculation, we deduce that
\begin{align*}
    \int_{\mathcal{D}_0}|u^{TF}|^4\,dx=m\int_{\mathcal{D}_0}|u^{TF}|^6\,dx.
\end{align*}
Combining \eqref{357}, we get
\begin{align}\label{347}
    e^{TF}(\mathcal{D}_0,1)=-\left(\frac{1}{4}-\frac{1}{6m}\right)\int_{\mathcal{D}_0}|u^{TF}|^4\,dx
\end{align}
and
\begin{gather}\label{348}
    \begin{aligned}
		\mu^{TF}=-\left(1-\frac{1}{m}\right)\int_{\mathcal{D}_0}|u^{TF}|^4\,dx
        .
    \end{aligned}
\end{gather}
\end{re}

\begin{lem}\label{lem41}
Assume that $d=1,2,3$, $\mathcal{D}_0=B(0,\sqrt[d]{d/(\rho\omega_d)})$ and $0<\rho\leq1$ is a fixed parameter. Then
\begin{gather*}
    \begin{aligned}
    0\leq e_L(\mathcal{D}_0,1)-e^{TF}(\mathcal{D}_0,1)\leq {O}(L^{-1})\quad\text{as }L\to+\infty.
    \end{aligned}
\end{gather*}	
Moreover, if $w_L$ is a minimizer of $ e_L(\mathcal{D}_0,1)$, it holds that
\begin{gather}\label{43}
	\begin{aligned}
		\int_{\mathcal{D}_0}|\nabla w_L|^2\,dx\leq O(L)\quad\text{as }L\to+\infty.
	\end{aligned}
\end{gather}
\end{lem}
\begin{pf}
From the proof of Lemma \ref{TLer}, \eqref{min1} and \eqref{45}, we have
\begin{gather*}
    \begin{aligned}
	0\leq e_L(\mathcal{D}_0,1)-e^{TF}(\mathcal{D}_0,1)\leq {O}(L^{-1})\quad\text{as }L\to+\infty.
    \end{aligned}
\end{gather*}	
In addition, if $w_L$ is a non-negative minimizer of $ e_L(\mathcal{D}_0,1)$, then
\begin{gather*}
    \begin{aligned}
    \frac{L^{-2}}{2(\rho|\mathcal{D}|)^{\frac{2}{d}}}\int_{\mathcal{D}_0}|\nabla w_L|^2\,dx&= e_L(\mathcal{D}_0,1)+\frac{1}{4}\int_{\mathcal{D}_0}|w_L|^{4}\,dx-\frac{1}{6}\int_{\mathcal{D}_0}|w_L|^6\,dx\\
    &=e_L(\mathcal{D}_0,1)-E^{TF}(w_L).
    \end{aligned}
\end{gather*}
By Lemma \ref{A} and  $u^{TF}$ is  a minimizer of $e^{TF}(\mathcal{D}_0,1)$. Thus, we obtain
\begin{gather*}
    \begin{aligned}
	\frac{L^{-2}}{2(\rho|\mathcal{D}|)^{\frac{2}{d}}}\int_{\mathcal{D}_0}|\nabla w_L|^2\,dx\leq e_L(\mathcal{D}_0,1)-e^{TF}(\mathcal{D}_0,1)\leq {O}(L^{-1})\quad\text{as }L\to+\infty.
    \end{aligned}
\end{gather*}
Therefore, \eqref{43} holds and the proof of this lemma is finished.
\end{pf}

\subsection{Proof of Theorem \ref{Thermodynamic-limit}}
In order to Theorem \ref{Thermodynamic-limit}, we need to give the asymptotic behavior of the Lagrange multiplier $\mu_L$. Now, we given the following Pohozaev-type identity, from Lemma \ref{BP}, for $\mathcal{D}_0=B(0,\sqrt[d]{d/(\rho\omega_d)})$, we have $\nu=x/|x|$   and 
\begin{gather}\label{Pb}
    \begin{aligned}
    \frac{L^{-2}}{2(\rho|\mathcal{D}|)^{\frac{2}{d}}}\int_{\partial\mathcal{D}_0}|x||\nabla   w_L|^2&\,dS+\frac{(d-2)L^{-2}}{2(\rho|\mathcal{D}|)^{\frac{2}{d}}}\int_{\mathcal{D}_0}|\nabla w_L|^2\,dx\\
    &=\frac{d}{4}\int_{\mathcal{D}_0}|w_L|^4\,dx-\frac{d}{6}\int_{\mathcal{D}_0}|w_L|^6\,dx+\frac{d\mu_L}{2}\int_{\mathcal{D}_0}|w_L|^2\,dx.
    \end{aligned}
\end{gather}
Moreover, we deduce  from \eqref{hEL} and \eqref{Pb} that
\begin{gather*}
	\begin{aligned}
	\frac{L^{-2}}{2d(\rho|\mathcal{D}|)^{\frac{2}{d}}}\int_{\partial\mathcal{D}_0}|x||\nabla w_L|^2&\,dS-\frac{L^{-2}}{d(\rho|\mathcal{D}|)^{\frac{2}{d}}}\int_{\mathcal{D}_0}|\nabla   w_L|^2\,dx\\
    &=\frac{1}{3}\int_{ \mathcal{D}_0}| w_L|^6\,dx-\frac{1}{4}\int_{ \mathcal{D}_0}| w_L|^4\,dx.
	\end{aligned}
\end{gather*}
Hence
\begin{gather}\label{325}
	\begin{aligned}
	 e_L(\mathcal{D}_0,1)= \frac{L^{-2}}{4d(\rho|\mathcal{D}|)^{\frac{2}{d}}}\int_{\partial\mathcal{D}_0}|x||\nabla   w_L|^2dS+\frac{(d-1)L^{-2}}{2d(\rho|\mathcal{D}|)^{\frac{2}{d}}}\int_{\mathcal{D}_0}|\nabla w_L|^2dx -\frac{1}{8}\int_{\mathcal{D}_0}|w_L|^4dx.
	\end{aligned}
\end{gather}
The corresponding Lagrange multiplier $\mu_L$ can be computed as
\begin{gather}\label{326}
	\begin{aligned}
		\mu_L= \frac{3L^{-2}}{2d(\rho|\mathcal{D}|)^{\frac{2}{d}}}\int_{\partial\mathcal{D}_0}|x||\nabla   w_L|^2\,dS+\frac{(d-3)L^{-2}}{d(\rho|\mathcal{D}|)^{\frac{2}{d}}}\int_{\mathcal{D}_0}|\nabla w_L|^2\,dx -\frac{1}{4}\int_{\mathcal{D}_0}|w_L|^4\,dx.
	\end{aligned}
\end{gather}

\begin{proof}[\textbf{Proof of Theorem \ref{Thermodynamic-limit}}:] Note that \eqref{TDL} in Theorem \ref{Thermodynamic-limit} comes from  Lemma \ref{TLer}.  Next, we prove \eqref{13} holds for $0<\rho\leq1$, $\mathcal{D}=B(0,R)$ and $\mathcal{D}_0=B(0,\sqrt[d]{d/(\rho\omega_d)})$.

We deduce from \eqref{similarly1} that
\begin{gather*}
	\begin{aligned}
		e^{TF}(\mathcal{D}_0,1)\leq-\frac{1}{4}\int_{\mathcal{D}_0}|w_L|^{4}\,dx+\frac{1}{6}\int_{\mathcal{D}_0}|w_L|^6\,dx\leq e^{TF}(\mathcal{D}_0,1)+O(L^{-1}),
	\end{aligned}
\end{gather*}
which means that $\{w_L\}$ is a minimizing sequence for $e^{TF}(\mathcal{D}_0,1)$. 
Note that $w_{L}$ is a non-negative radially symmetric decreasing  function, and $\{w_{L}\}$ is bounded in $L^2(\mathcal{D}_0)$ and $L^6(\mathcal{D}_0)$, see \eqref{29}. Then, we have the following estimate
 \begin{gather}\label{w-tau-b}
    \begin{aligned}
    w_L(|x|)\leq U(|x|)=:C\min\left\lbrace |x|^{-\frac{d}{2}}, |x|^{-\frac{d}{6}}\right\rbrace\quad \text{for all } 0<|x|<\sqrt[d]{d/(\rho\omega_d)}.
    \end{aligned}
 \end{gather}
Indeed, for any $x\in\mathcal{D}_0$ and $x\not=0$ and $s\in[2,6]$,
\begin{gather}\notag
    \begin{aligned}
    |w_L(x)|^s|x|^d=d\int_{0}^{|x|}|w_L(x)|^sr^{d-1}\,dr= d\int_{0}^{|x|}|w_L(r)|^sr^{d-1}\,dr\leq C\int_{B_{|x|}(0)}|w_L(x)|^s\,dx\leq C.
    \end{aligned}
\end{gather}
By Helly’s selection theorem and \eqref{w-tau-b}, for monotone functions, there exists a non-negative radially symmetric decreasing function $w_0(|x|)\leq U(|x|)$ such that, up to a subsequence,
\begin{gather*}
	\begin{aligned}
		w_L(x)\to w_0(x)\quad  \text{a.e. in }\mathcal{D}_0\quad\text{as }L\to+\infty.
		\end{aligned}
	\end{gather*}
Notice that $U\in L^p(\mathcal{D}_0)$ for all $p\in(2,6)$, by the Lebesgue’s dominated convergence, we conclude that for every $p\in(2,6)$,
\begin{gather}\label{3651}
	\begin{aligned}
	\lim_{L\to+\infty}\int_{\mathcal{D}_0}|w_L|^p\,dx=\int_{\mathcal{D}_0}|w_0|^p\,dx.
	\end{aligned}
\end{gather}
Combining the lower semicontinuity of $L^6$-norm, we get
\begin{gather}\label{Ho3}
	\begin{aligned}
	E^{TF}(w_0)\leq  e^{TF}(\mathcal{D}_0,1).
		\end{aligned}
	\end{gather}
Next, we prove $\int_{\mathcal{D}_0}|w_0|^2\,dx\equiv1$. Let $t:=\int_{\mathcal{D}_0}|w_0|^2\,dx$ and $0\leq t<1$. For $0<\rho\leq3/4$, by the H\"older's inequality and Young's inequality, we get 
\begin{gather}\label{}
    \begin{aligned}
    \int_{\mathcal{D}_0}|w_0|^4\,dx
    \leq\frac{3}{8}\int_{\mathcal{D}_0}|w_0|^2\,dx+\frac{2}{3}\int_{\mathcal{D}_0}|w_0|^6\,dx.
    \end{aligned}
\end{gather}
Hence, by \eqref{Ho3} and Lemma \ref{A}, we have
\begin{gather}\label{}
	\begin{aligned}
	e^{TF}(\mathcal{D}_0,1)\geq E^{TF}(w_0)\geq -\frac{3}{32}t> -\frac{3}{32}= e^{TF}(\mathcal{D}_0,1),
	\end{aligned}
\end{gather}
which is a contradiction.

For $3/4<\rho<1$,  by the H\"older's inequality, we get
\begin{gather*}
    \begin{aligned}
   \int_{\mathcal{D}_0}|w_0|^4\,dx\leq\left(\int_{\mathcal{D}_0}\,dx\right)^{\frac{4\rho-3}{3}}\left(\int_{\mathcal{D}_0}|w_0|^2\,dx\right)^{2(1-\rho)}\left(\int_{\mathcal{D}_0}|w_0|^6\,dx\right)^{\frac{2\rho}{3}}.
    \end{aligned}
\end{gather*}
From $|\mathcal{D}_0|=\rho^{-1}$ and Young's inequality, we get
\begin{gather}\label{Ho4}
    \begin{aligned}
    \int_{\mathcal{D}_0}|w_0|^4\,dx&\leq \left(\rho^{-1}\right)^{\frac{4\rho-3}{3}}\left(\int_{\mathcal{D}_0}|w_0|^2\,dx\right)^{2(1-\rho)}\left(\int_{\mathcal{D}_0}|w_0|^6\,dx\right)^{\frac{2\rho}{3}}\\
    &\leq\frac{3-2\rho}{3}\rho\left(\int_{\mathcal{D}_0}|w_0|^2\,dx\right)^{\frac{6(1-\rho)}{3-2\rho}}+\frac{2}{3}\int_{\mathcal{D}_0}|w_0|^6\,dx.
    \end{aligned}
\end{gather}
Note that $|\Omega|=\rho^{-1}$ for $3/4<\rho<1$. Hence, combining \eqref{Ho3}, \eqref{Ho4} and Lemma \ref{A}, we have
\begin{gather}\label{}
	\begin{aligned}
	e^{TF}(\mathcal{D}_0,1)\geq E^{TF}(w_0)\geq -\frac{3-2\rho}{12}\rho t^{\frac{6(1-\rho)}{3-2\rho}}> -\left(\frac{1}{2}-\frac{1}{3|\Omega|}\right)\frac{1}{2|\Omega|}= e^{TF}(\mathcal{D}_0,1),
	\end{aligned}
\end{gather}
which is a contradiction. 
Thus, $\int_{\mathcal{D}_0}|w_0|^2\,dx\equiv1$ and $w_0$ is a minimizer of $e^{TF}(\mathcal{D}_0,1)$.

Note that $w_0$ is a non-negative radially symmetric decreasing function. Using Lemma \ref{A},  the minimizer of $e^{TF}(\mathcal{D}_0,1)$ is unique when we restrict $u^{TF}$ to a radially symmetric non-increasing function. Thus, $w_0= u^{TF}$. Therefore, we get that
\begin{align}\label{3555}
w_{L_k}(x) \to u^{TF}(x) \ \ \text{strongly in}\  L^2\cap L^6(\mathcal{D}_0).
\end{align}	
Thus,  \eqref{115} holds and we complete the proof of Theorem \ref{Thermodynamic-limit}.
\end{proof}
\subsection{Proof of Theorem \ref{convergence}}
In this subsection, we aim to prove Theorem \ref{convergence}. First, we show the following interpolation estimates that in the spirit of the Gagliardo-Nirenberg inequality (see e.g. \cite{Nirenbergelliptic} or \cite[Lemma A.1]{CV1993Asymptotics}).
\begin{pro}\label{nablees}
Assume $u$ satisfies
\begin{align*}
	-\Delta u=f\quad\text{in}\ A\subset\mathbb{R}^d.
\end{align*}
Then
\begin{align*}
	|\nabla u(x)|^2\leq
	C\left(\|f\|_{L^\infty(A)}\|u\|_{L^\infty(A)}+\frac{1}{dist^2(x,\partial A)}\|u\|_{L^\infty(A)}^2 \right)
	\quad\text{in}\ A\subset\mathbb{R}^d,
\end{align*}
where $C$ is a constant depending only on $d$.
\end{pro}

\begin{lem}\label{lem44}
Assume $0<\rho<3/4$, $\mathcal{D}=B(0,R) \subset \mathbb{R}^{d}$  $(R>0$ and $d=1,2,3)$ and $L>0$. As Theorem \ref{Thermodynamic-limit}, let   $\mathcal{D}_L=B(0,LR)$, $\Omega=B(0,\sqrt[d]{4d/(3\omega_d)})$,  and $\mathcal{D}_0=B(0,\sqrt[d]{d/(\rho\omega_d)})$.
Suppose that $w_L(x)$ is a non-negative radially symmetric decreasing minimizer of $ e_L(\mathcal{D}_0,1)$, 
we then have
\begin{align}\label{60}
	\|\Delta w_L(x)\|_{L^\infty\left(\Omega\right)}\leq O\big(L^{\frac{4}{3}}\big)\quad\text{as}\ L\to+\infty.
\end{align}
\end{lem}
\begin{pf}
Rewrite the equation \eqref{hEL} as
\begin{align}\label{61}
	-\Delta  w_L(x)=f(x)\quad\text{in}\ \mathcal{D}_0=B\left(0,\sqrt[d]{d/(\rho\omega_d)}\right),
\end{align}
where
\begin{gather*}
	\begin{aligned}
		f(x):=(\rho|\mathcal{D}|)^{\frac{2}{d}}L^{2}\big[w_L^3(x)-w_L^5(x)+\mu_L w_L(x)\big].
	\end{aligned}
\end{gather*}
Note that $w_{L}$ is a radially symmetric decreasing  function, thus $0$ is the maximum point of $w_L$.  From  \eqref{29} that
	\begin{gather*}
		\begin{aligned}
			1\lesssim\int_{\mathcal{D}_0}|w_L|^4\,dx\leq w^{2}_L(0)\int_{\mathcal{D}_0}|w_L|^2\,dx= w^{2}_L(0).
		\end{aligned}
	\end{gather*}
	This implies that
	\begin{gather}\label{wtau1}
		\begin{aligned}
			w_L(0)\gtrsim 1.
		\end{aligned}
	\end{gather}
On the other hand, since $0$ is a minimax point of $w_L$, we then have  $-\Delta w_L(0)\geq0$. From \eqref{hEL} and \eqref{339}, we get
	\begin{gather*}
		\begin{aligned}
			w_L^5(0)\leq w_L^3(0),
		\end{aligned}
	\end{gather*}
which implies
	\begin{gather}\label{371}
		\begin{aligned}
			w_L(0)\leq1.
		\end{aligned}
	\end{gather}
Then, \eqref{wtau1} and \eqref{371} yields
	\begin{gather}\label{366}
		\begin{aligned}
			\|w_L\|_{L^\infty(\mathcal{D}_0)}\sim1.
		\end{aligned}
	\end{gather}
Combining Lemma \ref{lem36}, we deduce that
\begin{gather}\label{343}
	\begin{aligned}
	\|f\|_{L^\infty(\mathcal{D}_0)}\lesssim L^{2}\quad\text{as}\ L\to+\infty.
	\end{aligned}
\end{gather}

Let  $r_0:=\sqrt[d]{(4d)/(3\omega_d)}$ and 
$$\delta_1:=\frac{\sqrt[d]{d/(\rho\omega_d)}-\sqrt[d]{(4d)/(3\omega_d)}}{4}>0.$$
Note that $B(0,r_0+3\delta_1)\subset\mathcal{D}_0=B\left(0,\sqrt[d]{d/(\rho\omega_d)}\right)$ for any $0<\rho<3/4$. Moreover, using \eqref{343} and Proposition \ref{nablees}, we have
		\begin{align}\label{453}
			|\nabla w_L(x)|^2\leq
			C\left(L^2+\frac{1}{dist^2\big(x,\partial B(0,r_0+2\delta_1)\big)} \right)
			\quad\text{in}\ B(0,r_0+3\delta_1),
		\end{align}
where $C>0$ is a constant independent of $L$. Therefore, for any $x\in B(0,r_0+2\delta_1)$, taking $L>\delta_1^{-1}$,  we obtain $L^{-1}< dist\big(x,\partial B(0,r_0+3\delta_1)\big)$. Thus, from \eqref{453} that, for any $x\in B(0,r_0+2\delta_1)$,
\begin{align}\label{63}
    |\nabla w_L(x)|\lesssim L\quad\text{as}\ L\to+\infty.
\end{align}
		
For any $x\in B(0,r_0+2\delta_1)$, by \eqref{61}, we get
\begin{align*}
    -\Delta(\nabla w_L)= (\rho|\mathcal{D}|)^{\frac{2}{d}}L^{2}\left(3w_L^2-5w_L^4+\mu_L \right)\nabla w_L
\end{align*}
It follow from Lemma \ref{lem36}, \eqref{366}  and \eqref{63} that
\begin{align*}
    \|(\rho|\mathcal{D}|)^{\frac{2}{d}}L^{2}\left(3w_L^2-5w_L^4+\mu_L \right)\nabla w_L\|_{L^\infty\big(B(0,r_0+2\delta_1)\big) }\lesssim L^{3}\quad\text{as}\ L\to+\infty.
\end{align*}
Again using Proposition \ref{nablees} we get
\begin{align*}
    \big|\nabla (\nabla w_L)\big|^2\leq C\left(L^3+\frac{1}{dist^2\big(x,\partial B(0,r_0+2\delta_1)\big)} \right)\leq C{L^3}\quad\text{in}\ B(0,r_0+2\delta_1).
\end{align*}
Hence, for any $x\in B(0,r_0+\delta_1)$, we obtain
\begin{align}\label{z1}
    |\Delta w_L(x)|\lesssim L^{\frac{3}{2}}=:L^{a_1}\quad\text{as}\ L\to+\infty.
\end{align}

Let $\delta_2:={\delta_1}/{4}>0$. Combining  \eqref{61} and \eqref{z1}, we deduce that
\begin{gather*}
	\begin{aligned}
	\|f\|_{L^\infty(B(0,r_0+\delta_1))}\lesssim L^{\frac{3}{2}}\quad\text{as}\ L\to+\infty.
	\end{aligned}
\end{gather*} 
By repeating the above process, for any $x\in B(0,r_0+\delta_2)$, we can obtain
\begin{align*}
    |\Delta w_L(x)|\lesssim L^{\frac{3}{8}+1}=:L^{a_2}\quad\text{as}\ L\to+\infty.
\end{align*}
Let $\delta_3:={\delta_2}/{4}>0$. Once again, by repeating the above process, for any $x\in B(0,r_0+\delta_3)$, we have
\begin{align*}
    |\Delta w_L(x)|\lesssim L^{\frac{\frac{3}{8}+1}{4}+1}=:L^{a_3}\quad\text{as}\ L\to+\infty.
\end{align*} 
Moreover, let $n=4,5,...$, and $\delta_n:={\delta_{n-1}}/{4}>0$, by repeating the above process, for any $x\in B(0,r_0+\delta_n)$, we get
\begin{align}\label{v1}
    |\Delta w_L(x)|\lesssim L^{a_n}=L^{\frac{a_{n-1}}{4}+1}\quad\text{as}\ L\to+\infty.
\end{align} 
Hence, we obtain iterative sequence $\{a_n\}$ satisfying  
$${a_n}={\frac{1}{4}a_{n-1}+1}\quad\text{and}\quad a_1=\frac{3}{2}.$$
Note that $\{a_n\}$ is an arithmetic sequence with respect to the positive integer $n$ and it satisfies
$$a_n=\frac{4}{3}+\frac{1}{6\times4^{n-1}}\to \frac{4}{3}.$$
Taking $n\to \infty$ in \eqref{v1}, then \eqref{60} holds and we complete the proof of Lemma \ref{lem44}.
\end{pf}

Next, we give the following relationship between $\mu_L$ and $\mu^{TF}$.
\begin{lem}\label{lem36}
Let $0<\rho<3/4$, we then have
	\begin{gather}\label{339}
		\begin{aligned}
			|\mu_L-\mu^{TF}|\leq {O}(L^{-1})\quad\text{as }L\to+\infty,
		\end{aligned}
	\end{gather}
where $\mu_L$ and $\mu^{TF}$ be defined by \eqref{hEL} and \eqref{357}, respectively.
\end{lem}
\begin{pf}
Using $w_L\in C^2$ (by the standard elliptic regularity theory) and \eqref{3555}, it follows that for $L>0$ is sufficiently large,
\begin{gather*}
		\begin{aligned}
			w_L(x)\to u^{TF}(x)\equiv0,\quad\text{for all } \sqrt[d]{(4d)/(3\omega_d)}<|x|<\sqrt[d]{d/(\rho\omega_d)}.
		\end{aligned}
	\end{gather*}
Hence, as $L\to+\infty$,
	\begin{gather}\label{391}
		\begin{aligned}
			|w_L^4(x)-w_L^2(x)|	\leq\frac{1}{16},\quad\text{for all } \sqrt[d]{(4d)/(3\omega_d)}<|x|<\sqrt[d]{d/(\rho\omega_d)}.
		\end{aligned}
	\end{gather}
On the other hand, we deduce from Lemma \ref{lem36} and  $\mu^{TF}=-{3}/{16}$ that $-4/16<\mu_L<-2/16$ for  $L>0$ is sufficiently large. It follow from \eqref{hEL} and \eqref{391}  that there exists a constant $\beta>0$ such that, as $L\to+\infty$,
	\begin{gather*}
		\begin{aligned}
			-\Delta  w_L(x)+\frac{\beta^2L^{2}}{16}w_L(x)
			\leq0,\quad \text{ for all } \sqrt[d]{(4d)/(3\omega_d)}<|x|<\sqrt[d]{d/(\rho\omega_d)}.
		\end{aligned}
	\end{gather*}
	Let $v\in L^{\infty}\Big( \mathbb{R}^d\setminus B(0,\sqrt[d]{(4d)/(3\omega_d)}),\mathbb{R}\Big) $ be such that
	\begin{gather*}
		\begin{cases}
			-\Delta v(x)+\frac{\beta^2L^{2}}{16}v(x)=0\ \ &\text{if}\ x\in\mathbb{R}^d\setminus B(0,\sqrt[d]{(4d)/(3\omega_d)}),\\
			v(x)=w_{L}(x)&\text{if}\ x\in B(0,\sqrt[d]{(4d)/(3\omega_d)}),\\
			\lim_{|x|\to\infty}v(x)=0.
		\end{cases}
	\end{gather*}
By \cite[Lemma 6.4]{JFAvm}, for every $|x|>\sqrt[d]{(4d)/(3\omega_d)}$, there exists a constant $C\in\mathbb{R}^+$ such that
\begin{align*}
    v(x)\leq C|x|^{-\frac{d-1}{2}}e^{-\frac{\beta L}{4}|x|}.
\end{align*}
Hence, by the comparison principle \cite{Okavian}, for every $\sqrt[d]{(4d)/(3\omega_d)}<|x|<\sqrt[d]{d/(\rho\omega_d)}$, we then have
\begin{align}\label{J21}
    w_L(x)\leq v(x)\leq {O}\left(  |x|^{-\frac{d-1}{2}}e^{-\frac{\beta L}{4}|x|}\right) \quad\text{as}\ L\to+\infty.
\end{align}
Thus, applying the local elliptic estimates (cf. (3.15) in \cite{GT97}) to \eqref{hEL}, we get
\begin{align*}
|\nabla w_L(x)|\leq {O}\left( |x|^{-\frac{d-1}{2}}e^{-\frac{\beta L}{8}|x|}\right) \quad\text{holds for sufficiently large } |x|>0\text{ as }L\to+\infty,
\end{align*}
which then implies that there exists a constant $C>0$ such that
\begin{gather}\label{b3}
	\begin{aligned}
		\int_{\partial\mathcal{D}_0}|x||\nabla w_L|^2\,dS\leq C Le^{-\frac{\beta L^{2}}{4}}\quad\text{as }L\to+\infty.
	\end{aligned}
\end{gather}

Finally, combining \eqref{347} and \eqref{325}, we then have
	\begin{gather*}
		\begin{aligned}
		\int_{\mathcal{D}_0}|u^{TF}|^4\,dx=-8e^{TF}(\mathcal{D}_0,1)
		\end{aligned}
	\end{gather*}
	and
\begin{gather*}
	\begin{aligned}
   \int_{\mathcal{D}_0}|w_L|^4dx=-8\left(e_L(\mathcal{D}_0,1)-\frac{L^{-2}}{4d(\rho|\mathcal{D}|)^{\frac{2}{d}}}\int_{\partial\mathcal{D}_0}|x||\nabla   w_L|^2dS-\frac{(d-1)L^{-2}}{2d(\rho|\mathcal{D}|)^{\frac{2}{d}}}\int_{\mathcal{D}_0}|\nabla w_L|^2dx\right) .
	\end{aligned}
\end{gather*}
From Lemma \ref{lem41} and \eqref{b3}, we obtain
\begin{gather*}
    \begin{aligned}
	&\quad\left| \int_{\mathcal{D}_0}|w_L|^4\,dx-\int_{\mathcal{D}_0}|u^{TF}|^4\,dx\right|\\
	&\leq8| e_L(\mathcal{D}_0,1)-e^{TF}(\mathcal{D}_0,1)|+\frac{L^{-2}}{4d(\rho|\mathcal{D}|)^{\frac{2}{d}}}\int_{\partial\mathcal{D}_0}|x||\nabla   w_L|^2dS+\frac{(d-1)L^{-2}}{2d(\rho|\mathcal{D}|)^{\frac{2}{d}}}\int_{\mathcal{D}_0}|\nabla w_L|^2dx\\
	& \leq {O}(L^{-1})\quad\text{as }L\to+\infty.
    \end{aligned}
\end{gather*}
We deduce from \eqref{348} and \eqref{326} that
\begin{gather*}
    \begin{aligned}
	|\mu_L-\mu^{TF}|&\leq\frac{1}{4}\left| \int_{\mathcal{D}_0}|w_L|^4\,dx-\int_{\mathcal{D}_0}|u^{TF}|^4\,dx\right|\\
 &\quad+\frac{3L^{-2}}{2d(\rho|\mathcal{D}|)^{\frac{2}{d}}}\int_{\partial\mathcal{D}_0}|x||\nabla   w_L|^2\,dS+\frac{(3-d)L^{-2}}{d(\rho|\mathcal{D}|)^{\frac{2}{d}}}\int_{\mathcal{D}_0}|\nabla w_L|^2\,dx\\
	& \leq {O}(L^{-1})\quad\text{as }L\to+\infty.
    \end{aligned}
\end{gather*}
Thus, we complete the proof of this lemma.
\end{pf}

\begin{proof}[\textbf{Proof of Theorem \ref{convergence}}:] As  \eqref{112} and \eqref{114} in Theorem \ref{convergence} following from  \eqref{60} and  \eqref{J21}, respectively. Next, we prove \eqref{113} for $\Omega=B(0,\sqrt[d]{4d/(3\omega_d)})$ and $\mathcal{D}_0=B(0,\sqrt[d]{d/(\rho\omega_d)})$. 

Observe that $w_L(x)\geq0$ and $u^{TF}(x)\geq0$ satisfy the following equation: 
\begin{gather*}
    \begin{aligned}
    -\frac{L^{-2}}{(\rho|\mathcal{D}|)^{\frac{2}{d}}}\Delta  w_L(x)-  w_L^3(x)+w_L^5(x)=\mu_L w_L(x)\ \ \text{in }\mathcal{D}_0,
    \end{aligned}
\end{gather*}
and
\begin{gather*}
    \begin{aligned}
    -(u^{TF}(x))^3+(u^{TF}(x))^5=\mu^{TF}u^{TF}(x)\ \ \text{in }\mathcal{D}_0,
\end{aligned}
\end{gather*}
respectively. Hence, for any $x\in\mathcal{D}_0$, we have
\begin{gather*}
    \begin{aligned}
     -\frac{L^{-2}}{(\rho|\mathcal{D}|)^{\frac{2}{d}}}\Delta  w_L- \left( w_L^3-(u^{TF})^{3}\right) +\left( w_L^5-(u^{TF})^{5}\right)
        &=\mu_L w_L-\mu^{TF}u^{TF}\\
        &=\left(\mu_L-\mu^{TF}\right)w_L+\mu^{TF}\left(w_L(x)-u^{TF}\right).
    \end{aligned}
\end{gather*}
Using the binomial theorem, we obtain
\begin{gather*}
    \begin{aligned}
	w_L^3-(u^{TF})^{3}&=\left(\left( w_L-u^{TF}\right) +u^{TF}\right)^{3}-\left(u^{TF}\right)^{3}\\
	&=\left( w_L-u^{TF}\right)^3
	+3\left( w_L-u^{TF}\right)^2u^{TF}
	+3\left( w_L-u^{TF}\right)(u^{TF})^2
    \end{aligned}
\end{gather*}
and
\begin{gather*}
    \begin{aligned}
	w_L^5-(u^{TF})^{5}&=\left(\left( w_L-u^{TF}\right) +u^{TF}\right)^{5}-(u^{TF})^{5}\\
        =&\left( w_L-u^{TF}\right)^5
	+5\left( w_L-u^{TF}\right)^4u^{TF}
	+10\left( w_L-u^{TF}\right)^3(u^{TF})^{2}\\
	&+10\left( w_L-u^{TF}\right)^2(u^{TF})^{3}
	+5\left( w_L-u^{TF}\right)(u^{TF})^{4}.
    \end{aligned}
\end{gather*}
Note that $\mu^{TF}=-3/16$ and $u^{TF}(x)\equiv\sqrt{{3}/{4}}$ in $x\in\Omega=B(0,\sqrt[d]{(4d)/(3\omega_d)})\subset\mathcal{D}_0=B(0,\sqrt[d]{d/(\rho\omega_d)})$. Moreover, for any $x\in\Omega$, we have
	\begin{gather*}
		\begin{aligned}
			\frac{3}{4}\left( w_L-u^{TF}\right) =&-\left( w_L-u^{TF}\right)^5
			-5\left( w_L-u^{TF}\right)^4u^{TF}
			-10\left( w_L-u^{TF}\right)^3(u^{TF})^{2}\\
			&-10\left( w_L-u^{TF}\right)^2(u^{TF})^{3}+\left( w_L-u^{TF}\right)^3
			+3\left( w_L-u^{TF}\right)^2u^{TF}\\
			&+\frac{L^{-2}}{(\rho|\mathcal{D}|)^{\frac{2}{d}}}\Delta  w_L+\left(\mu_L-\mu^{TF}\right)  w_L.
		\end{aligned}
	\end{gather*}
Therefore, there exists a subsequence $\{L_k\}$ such that, for $L_k$ is sufficiently large, from  \eqref{366}, \eqref{339} and Lemma \ref{lem44} that
	\begin{gather*}
		\begin{aligned}
			\left\|w_{L_k}-u^{TF}\right\|_{L^\infty\left( \Omega\right)} &\lesssim L_k^{-\frac{2}{3}}
			+\left\|w_{L_k}-u^{TF}\right\|^2_{{L}^\infty\left( \Omega\right)}.
		\end{aligned}
	\end{gather*}
Thus, as ${L_k}\to+\infty$,
	\begin{gather*}
		\begin{aligned}
			\left\|w_{L_k}-u^{TF}\right\|_{{L}^\infty\left( \Omega\right) }&\lesssim\frac{1-\sqrt{1-4C{L}_k^{-\frac{2}{3}}}}{2}\\
			&=\frac{4C{L}_k^{-\frac{2}{3}}}{2\left(1+\sqrt{1-4C{L}_k^{-\frac{2}{3}}} \right) }\\
			&\lesssim {L}_k^{-\frac{2}{3}}.
		\end{aligned}
	\end{gather*}
This implies that \eqref{113} holds, and we complete the proof of Theorem \ref{convergence}.
\end{proof}

\section{Thomas-Fermi limit}\label{sec4}

In this section, we aim to prove Theorem \ref{th14} for $\mathcal{D}=\mathbb{R}^d$. From Proposition \ref{pro1}, it is known that the minimizer $\varPhi_N$ of $e(\mathbb{R}^d,N)$ (see \eqref{minN}) is a $C^2(\mathbb{R}^d)$ function, non-negative, unique, radially symmetric decreasing, and satisfies the following Euler-Lagrange equation
\begin{gather*}
	\begin{aligned}
	-\Delta \varPhi_N-\varPhi_N^3+\varPhi_N^5=\mu_N \varPhi_N,\quad x\in\mathbb{R}^d,
	\end{aligned}
\end{gather*}
where $\mu_{N}\in\mathbb{R}$ is the corresponding Lagrange multiplier. 
	
\begin{lem}\label{lem21}
Assume that $d=1,2,3$ and $\varPhi_N$ be a minimizer of $e(\mathbb{R}^d,N)$, then, we have
\begin{align}\label{41}
    e(\mathbb{R}^d,N)\sim-N\quad\text{as }  N\to+\infty,
\end{align}
and
\begin{align}\label{4222}
    \int_{\mathbb{R}^d}|\varPhi_N|^4\,dx\sim \int_{\mathbb{R}^d}|\varPhi_N|^6\,dx\sim N\qquad\text{as}\ N\to+\infty.
\end{align}
\end{lem}	
\begin{pf} Lower bound estimate of \eqref{41}. For any $v\in H^1(\mathbb{R}^d)$ with $\|v\|_{L^2(\mathbb{R}^d)}^2=N$, using the H\"older's inequality and Young's inequality, we have
\begin{gather*}
    \begin{aligned}
    {E}(v)
    \geq-\frac{1}{4}\int_{\mathbb{R}^d}|v|^4\,dx+\frac{1}{6}\int_{\mathbb{R}^d}|v|^6\,dx\geq-\frac{3N}{32}.
    \end{aligned}
\end{gather*}
		
Next, we will give the upper bound estimate of the energy $e(\mathbb{R}^d,N)$ as $ N\to+\infty$. Let $Q_d$ be the unique positive radial ground state solution of \eqref{Qeq}. For any $\eta>0$, choose a test function
		\begin{gather*}
			\begin{aligned}
				Q_{\eta}(x)=\frac{N^{\frac{1}{2}}\eta^{\frac{d}{2}}}{\|Q_d\|_{2}} Q_d(\eta x),\quad x\in\mathbb{R}^d,
			\end{aligned}
		\end{gather*}
such that $\|Q_{\eta}\|_{2}^2=N$. By \eqref{Qid}, we get
\begin{gather}\label{291}
\begin{aligned}
E(Q_{\eta})&=\frac{1}{2}\int_{\mathbb{R}^d}|\nabla Q_{\eta}|^2\,dx-\frac{1}{4}\int_{\mathbb{R}^d}|Q_{\eta}|^4\,dx+\frac{1}{6}\int_{\mathbb{R}^d}|Q_{\eta}|^6\,dx\\
&=\frac{N\eta^2}{2}-\frac{N^2\eta^{d}}{4\|Q_d\|_{2}^2}+\frac{N^3\eta^{2d}}{6\|Q_d\|_{2}^6}\int_{\mathbb{R}^d}|Q_{d}|^6\,dx.
\end{aligned}
\end{gather}	
Let $\eta=C_0N^{-1/d}$ and $C_0>0$ is sufficiently small, such that
\begin{gather}\label{210}
	\begin{aligned}
		\frac{C_0^{d}}{4\|Q_d\|_{2}^2}-\frac{C_0^{2d}}{6\|Q_d\|_{2}^6}\int_{\mathbb{R}^d}|Q_{d}|^6\,dx>0.
	\end{aligned}
\end{gather}
Hence, we deduce from \eqref{291} and \eqref{210}  that
		\begin{gather*}
			\begin{aligned}
				e(N)\leq{E}(Q_{\eta})\lesssim N(N^{-2/d}-1)\lesssim-N\qquad\text{as}\ N\to+\infty.
			\end{aligned}
		\end{gather*}	
Thus, we completes the proof of \eqref{41}.

Finally, to prove \eqref{4222}. In fact, by the similar argument as in   Lemma \ref{lem22}, one can also obtain  \eqref{4222}. Here, we omit it.
\end{pf}

Inspired by the proof of Thermodynamic limit, we define 
\begin{align*}
    w_N(x)=\varPhi_N({N^{1/d}}{x}),\quad x\in\mathbb{R}^d.
\end{align*}
Note that $\|w_N\|_{2}^2=1$. We deduce from Lemma \ref{lem21} that
\begin{gather*}
    \begin{aligned}
    \int_{\mathbb{R}^d}|w_N|^4\,dx\sim \int_{\mathbb{R}^d}|w_N|^6\,dx\sim1\qquad\text{as}\ N\to+\infty.
    \end{aligned}
\end{gather*}
Now, we introduce the following constraint minimization problem
\begin{gather*}
	\begin{aligned}
	e_N(\mathbb{R}^d,1):=\inf\left\lbrace E_N(w): w\in H^1(\mathbb{R}^d), \int_{\mathbb{R}^d}|w|^2\,dx=1\right\rbrace,
	\end{aligned}
\end{gather*}
where the energy functional ${E}_{N}$ be given by
\begin{gather*}
    \begin{aligned}
	{E}_{N}(w)=\frac{N^{-2/d}}{2}\int_{ \mathbb{R}^d}|\nabla w|^2\,dx-\frac{1}{4}\int_{\mathbb{R}^d}|w|^{4}\,dx+\frac{1}{6}\int_{ \mathbb{R}^d}|w|^6\,dx.
    \end{aligned}
\end{gather*}
 Thus, $w_N$ is also non-negative radially symmetric ground state of $e_N(\mathbb{R}^d,1)$, i.e.,
\begin{gather*}
    \begin{aligned}
	N^{-1}E(\varPhi_N)={E}_{N}(w_N)=e_N(\mathbb{R}^d,1).
    \end{aligned}
\end{gather*}
And $w_N$ satisfies the following Euler-Lagrange equation
\begin{gather}\label{l3}
    \begin{aligned}
    -N^{-2/d}\Delta  w_N(x)-  w_N^3(x)+w_N^5(x)
    =\mu_N w_N(x),\quad x\in\mathbb{R}^d,
    \end{aligned}
\end{gather}
where $\mu_N \in\mathbb{R}$ is a Lagrange multiplier associated to $w_N(x)$.

In order to obtain the asymptotic behavior of the Lagrange multiplier $\mu_N$,  we need to give the following Pohozaev-type identity  $\partial_\eta E_{N}(\eta^{\frac{d}{2}} w_N(\eta x))|_{\eta=1}=0$, i.e.,
\begin{align*}
	N^{-2/d}\int_{\mathbb{R}^d}|\nabla w_N|^2\,dx
	+\frac{d}{3}\int_{\mathbb{R}^d}|w_N|^6\,dx=\frac{d}{4}\int_{\mathbb{R}^d}|w_N|^4\,dx.
\end{align*}
Combining \eqref{hEL}, it is known that the Lagrange multiplier $\mu_N$ can be computed as
\begin{gather}\label{hp1}
	\begin{aligned}
		\mu_N=\frac{(d-3)N^{-2/d}}{d}\int_{\mathbb{R}^d}|\nabla w_N|^2\,dx-\frac{1}{4}\int_{\mathbb{R}^d}|w_N|^4\,dx
	\end{aligned}
\end{gather}
and 
\begin{gather}\label{hp2}
\begin{aligned}
e_N(\mathbb{R}^d,1)=\frac{(d-4)N^{-2/d}}{2d}\int_{\mathbb{R}^d}|\nabla w_N|^2\,dx-\frac{1}{8}\int_{\mathbb{R}^d}|w_N|^4\,dx.
\end{aligned}
\end{gather}

\begin{pro}\label{p42}
Suppose that $d=1,2,3$ and $w_N$ is a minimizer of $e_N(\mathbb{R}^d,1)$. Then
\begin{itemize}
    \item[\rm(I)] (Total energy) $0\leq e_N(\mathbb{R}^d,1)-e^{TF}(\mathbb{R}^d,1)\leq {O}(N^{-1/d})$ as $ N\to+\infty$.
    \item[\rm(II)] (Kinetic energy) $\int_{\mathbb{R}^d}|\nabla w_N|^2\,dx\leq O(N^{1/d})$ as $ N\to+\infty$.
     \item[\rm(III)] (Lagrange multiplier) $|\mu_N-\mu^{TF}|\leq O(N^{-1/d})$ as $ N\to+\infty$.
      \item[\rm(IV)] (Interior Laplace estimate) $\|\Delta w_N(x)\|_{L^\infty(B(0,\sqrt[d]{(4d)/(3\omega_d)}))}\leq O\big(N^{\frac{4}{3d}}\big)$ as $ N\to+\infty$.
\end{itemize}
\end{pro}
\begin{pf}
Similar to the proof of Lemma \ref{lem41} and \ref{lem44}, from the above equations \eqref{l3}--\eqref{hp2}, we can obtain (I), (II) and (IV). Next, to prove (III) holds.

Combining \eqref{347} and \eqref{hp2}, we then have
\begin{gather*}
	\begin{aligned}
	\int_{\mathbb{R}^d}|u^{TF}|^4\,dx=-8	e^{TF}(\mathbb{R}^d,1)
	\end{aligned}
\end{gather*}
and
\begin{gather*}
	\begin{aligned}
	\int_{\mathbb{R}^d}|w_N|^4\,dx=-8\left(  e_N(\mathbb{R}^d,1)+\frac{(d-4)N^{-2/d}}{2d}\int_{\mathbb{R}^d}|\nabla w_N|^2\,dx\right) .
	\end{aligned}
\end{gather*}
Thus, by (I) and (II), we obtain
\begin{gather}\label{424}
	\begin{aligned}
	\left| \int_{\mathbb{R}^d}|w_N|^4\,dx-\int_{\mathbb{R}^d}|u^{TF}|^4\,dx\right|
	&\leq8| e_N(\mathbb{R}^d,1)-e^{TF}(\mathbb{R}^d,1)|+12N^{-2/d}\int_{\mathbb{R}^d}|\nabla w_N|^2\,dx\\
	& \leq {O}(N^{-1/d})\quad\text{as }N\to+\infty.
	\end{aligned}
\end{gather}
We deduce from \eqref{348}, \eqref{hp1} and  \eqref{424} that
\begin{gather*}
	\begin{aligned}
	|\mu_N-\mu^{TF}|&=\left|\frac{1}{8}\int_{\mathbb{R}^d}|u^{TF}|^6\,dx-\frac{1}{8}\int_{\mathbb{R}^d}|w_N|^6\,dx+\frac{(d-3)N^{-2/d}}{d}\int_{\mathbb{R}^d}|\nabla w_N|^2\,dx\right| \\
	&\leq\frac{1}{8}\left| \int_{\mathbb{R}^d}|w_N|^6\,dx-\int_{\mathbb{R}^d}|u^{TF}|^6\,dx\right| +N^{-2/d}\int_{\mathbb{R}^d}|\nabla w_N|^2\,dx\\
	& \leq {O}(N^{-1/d})\quad\text{as }N\to+\infty.
	\end{aligned}
\end{gather*}
This implies that (III) holds, and we complete the proof of Proposition \ref{p42}.
\end{pf}

\begin{proof}[\textbf{Proof of Theorem \ref{th14}}:]
From \eqref{45} and Proposition \ref{p42}.(I), we get \eqref{116} holds. Similar to the proof of \eqref{13} and Theorem \ref{convergence}, we can obtain \eqref{115}, \eqref{10} and \eqref{11}. Here, we omit it.
\end{proof}

\appendix
\section{Appendix}
\subsection{The  Thomas-Fermi Problem}

In this  appendix, we present a simple argument of Thomas-Fermi minimizer $u^{TF}$, the limit profiles of $w_L$ above. Let $A\subseteq\mathbb{R}^d$ be a domain, which can be a bounded region, an unbounded region, or the whole space $\mathbb{R}^d$. Next, we consider the following  Thomas-Fermi minimization problem
\begin{gather}\label{TF1}
    \begin{aligned}
	e^{TF}(A,1)=\inf\left\lbrace {E}^{TF}(u):  u\in L^6(A)\text{ and } \int_{A}|u|^2\,dx=1\right\rbrace,
    \end{aligned}
\end{gather}
where the Thomas-Fermi energy functional $E^{TF}$ is given by
\begin{gather*}
    \begin{aligned}
	E^{TF}(u)=-\frac{1}{4}\int_{A}|u|^{4}\,dx+\frac{1}{6}\int_{A}|u|^6\,dx.
    \end{aligned}
\end{gather*}


\begin{lem}\label{A}[Thomas-Fermi has constant optimizers] Assume that $A\subseteq\mathbb{R}^d$ is a domain with $|A|\geq1$. Then there exists at least one minimizer of \eqref{TF1}, and 
$$e^{TF}(A,1)=-\left(\frac{1}{2}-\frac{1}{3m}\right)\frac{1}{2m}$$
with equality if and only if 
$$u^{TF}(x)=\sqrt{\frac{\mathds{1}_{\Omega}}{m}}(x)\quad\text{in }A$$
for some Borel set $\Omega\subseteq A$ and measure $|\Omega|=m =\min\{|A|, 4/3\}$. 
\end{lem}

\begin{re}  
Note that, if $0<\rho\leq1$, ground states $u^{TF}$ of the Thomas-Fermi minimization problem \eqref{minTF} are not unique at all, even up to translations, since for any set $\Omega\subset\mathcal{D}_0$ only the measure $|\Omega| =\min\{|A|, 4/3\}$ is required.
\end{re}
\begin{re}
For $0<|A|<1$, things  become more  complicated,  we have no idea.
\end{re}

\begin{proof}[\textbf{Proof of Lemma \ref{A}}] For $|A|\geq 4/3$, by the same argument as in \cite{ARMAGontier}, one can obtain that the Thomas-Fermi energy $e^{TF}(A,1)$ has a ground state. However, for $1\leq|A|<4/3$, we will give a new method to prove this. 

Now, we consider the following equivalence problem (i.e., let $\varphi(x)=|u(x)|^2$)
\begin{gather*}
    \begin{aligned}
		\mathop{\inf}_{\varphi(x)\geq0
		\atop
		\int_{  A}\varphi(x)\,dx=1
		}\int_{ A}\left(\frac{1}{6}\varphi^3(x)-\frac{1}{4}\varphi^2(x)\right)\,dx.
    \end{aligned}
\end{gather*}

\textbf{Case 1:} $|A|\geq4/3$.
For any $\varphi\geq0$ and $\int_{ A}\varphi\,dx=1$, using the H\"older's inequality and Young's inequality, we obtain
\begin{gather*}
    \begin{aligned}
    \int_{ A}\varphi^2(x)\,dx&\leq\left(\int_{ A}\varphi(x)\,dx\right)^{\frac{1}{2}}\left(\int_{ A}\varphi^3(x)\,dx\right)^{\frac{1}{2}}
    \leq\frac{3}{8}+\frac{2}{3}\int_{ A}\varphi^3(x)\,dx.
    \end{aligned}
\end{gather*}
Hence, we get
\begin{gather*}
    \begin{aligned}
    \mathop{\inf}_{\varphi(x)\geq0
		\atop\int_{  A}\varphi(x)\,dx=1
		}\int_{ A}\left(\frac{1}{6}\varphi^3(x)-\frac{1}{4}\varphi^2(x)\right)\,dx\geq-\frac{3}{32}.
    \end{aligned}
\end{gather*}
Let $\alpha_*:=\frac{1}{6}\rho_*^2-\frac{1}{4}\rho_*$, where
\begin{gather*}
    \begin{aligned}
        \rho_*=\mathop{\rm{argmin}}_{\varphi(x)>0}\left( \frac{1}{6}\varphi^2(x)-\frac{1}{4}\varphi(x)\right)=\frac{3}{4}.
    \end{aligned}
\end{gather*}
Defined a map
\begin{gather*}
    \begin{aligned}
		\varphi\in\mathbb{R}^+\mapsto \frac{1}{6}\varphi^3(x)-\frac{1}{4}\varphi^2(x)-\alpha_*\varphi(x)
    \end{aligned}
\end{gather*}
such that this map is non-negative over $\mathbb{R}^+$. Thus we have
\begin{gather*}
    \begin{aligned}
		&\int_{ A}\left(\frac{1}{6}\varphi^3(x)-\frac{1}{4}\varphi^2(x)\right)\,dx-\alpha_*\\
  &=\int_{ A}\left(\frac{1}{6}\varphi^3(x)-\frac{1}{4}\varphi^2(x)-\alpha_*\varphi(x)\right)\,dx\geq0
    \end{aligned}
\end{gather*}
with equality if and only if $\varphi$ takes only the two values $0$ and $\rho_*$. Since $| A|\geq4/3$, hence is of the form $\varphi(x)=\rho_*\mathds{1}_{\Omega}(x)$, where $\Omega\subseteq A$ is a Borel set with $|\Omega|=1/\rho_*$. The minimum in the statement is thus equal to $\alpha_*$ and its value follows after a computation.

\textbf{Case 2:} $4/3>|A|\geq1$.
For any $\varphi\geq0$ and $\int_{ A}\varphi\,dx=1$, using the H\"older's inequality, we obtain
\begin{gather*}
    \begin{aligned}
    \int_{A}\varphi^2(x)\,dx&=\int_{ A}1\cdot\varphi^\alpha(x)\cdot\varphi^\beta(x)\,dx\\
    &\leq\left(\int_{A}\,dx\right)^{\frac{1}{p_1}}\left(\int_{ A}\varphi^{\alpha p_2}(x)\,dx\right)^{\frac{1}{p_2}}\left(\int_{ A}\varphi^{\beta p_3}(x)\,dx\right)^{\frac{1}{p_3}},
    \end{aligned}
\end{gather*}
where the constants $\alpha,\beta\geq0$ and $p_1,p_2,p_3>1$ satisfy
\begin{gather*}
    \begin{cases}
        \frac{1}{p_1}+\frac{1}{p_2}+\frac{1}{p_3}=1,\\
        \qquad\quad\ \ \alpha p_2=1,\\
        \qquad\quad\ \ \beta p_3=3,\\
        \qquad\ \ \alpha +\beta=2.
    \end{cases}
\end{gather*}
Then,  we have
\begin{gather*}
    \begin{cases}
        \frac{1}{p_1}=1-\alpha-\frac{\beta}{3}=\frac{2\beta-3}{3}<1, (\text{since } p_1>1)\\
        \  \beta=2-\alpha\geq0.
    \end{cases}
\end{gather*}
Thus,  $\beta$ must be
$$3/2<\beta\leq2,$$
and
\begin{gather*}
    \begin{aligned}
    \int_{ A}\varphi^2(x)\,dx\leq\left(\int_{A}\,dx\right)^{\frac{2\beta-3}{3}}\left(\int_{ A}\varphi(x)\,dx\right)^{\frac{1}{p_2}}\left(\int_{ A}\varphi^3(x)\,dx\right)^{\frac{\beta}{3}}.
    \end{aligned}
\end{gather*}
From $3/2<\beta\leq2$, $\int_{ A}\varphi\,dx=1$ and Young's inequality, we get
\begin{gather*}
    \begin{aligned}
    \int_{ A}\varphi^2(x)\,dx
    &\leq\frac{3-\beta}{3}\left(\frac{\beta}{2}\right)^{\frac{3}{3-\beta}}|A|^{\frac{2\beta-3}{3-\beta}}+\frac{2}{3}\int_{ A}\varphi^3(x)\,dx.
    \end{aligned}
\end{gather*}
Hence,
\begin{gather}\label{A3}
    \begin{aligned}
    \mathop{\inf}_{\varphi(x)\geq0
		\atop\int_{  A}\varphi(x)\,dx=1
		}\int_{ A}\left(\frac{1}{6}\varphi^3(x)-\frac{1}{4}\varphi^2(x)\right)\,dx\geq-\frac{3-\beta}{3}\left(\frac{\beta}{2}\right)^{\frac{3}{3-\beta}}|A|^{\frac{2\beta-3}{3-\beta}}.
    \end{aligned}
\end{gather}
On the other hand, taking a test function $\varphi_0=|A|^{-1}$ in $A$, we have
\begin{gather}\label{A4}
    \begin{aligned}
	\int_{A}\left(\frac{1}{6}\varphi_0^3(x)-\frac{1}{4}\varphi_0^2(x)\right)\,dx=\left(\frac{1}{6|A|}-\frac{1}{4}\right)\frac{1}{|A|}.
    \end{aligned}
\end{gather}
Taking $\beta=2/|A|$, we get
\begin{gather}\label{A10}
    \begin{aligned}
    -\frac{3-\beta}{3}\left(\frac{\beta}{2}\right)^{\frac{3}{3-\beta}}|A|^{\frac{2\beta-3}{3-\beta}}&=-\frac{3-2/|A|}{3}\left(\frac{2/|A|}{2}\right)^{\frac{3}{3-2/|A|}}|A|^{\frac{4/|A|-3}{3-2/|A|}}\\
    &=\left(\frac{1}{6|A|}-\frac{1}{4}\right)\frac{1}{|A|}.
    \end{aligned}
\end{gather}
Since $3/2<\beta\leq2$, then \eqref{A10} holds if and only if $1\leq|A|<4/3$. Thus, we deduce from \eqref{A3}--\eqref{A10} that $\varphi_0=|A|^{-1}$ is a minimizer of $e^{TF}(A,1)$ for $4/3>|A|\geq1$.

Let $u^{TF}$ be a minimizer of $e^{TF}(A,1)$ for $4/3>|A|\geq1$. And $u^{TF}$ satisfies the following Euler-Lagrange equation
\begin{gather*}
	\begin{aligned}
	-(u^{TF})^3+(u^{TF})^5=\mu^{TF} u^{TF},\quad x\in A,
	\end{aligned}
\end{gather*}
where $\mu^{TF}\in\mathbb{R}$ is a Lagrange multiplier associated to $u^{TF}$. According to the Euler-Lagrange equation and the normalization condition, the following two cases hold:

(I) $u^{TF}\equiv|A|^{-1/2}$.

(II) $u^{TF}=((1-\gamma)|A|)^{-1/2}$ in $A\setminus B$, where $B$ is a  support set  of $u^{TF}$ and $|B|=\gamma|A|$ for $0\leq\gamma<1$.  

From the proof, it can be seen that (I) holds. Next, let's prove (II).
In fact, 
\begin{gather*}
    \begin{aligned}
	E^{TF}(u^{TF})&=-\frac{1}{4}\int_{A}|u^{TF}|^{4}\,dx+\frac{1}{6}\int_{A}|u^{TF}|^6\,dx\\
    &=\frac{1}{6}\frac{1}{(1-\gamma)^2|A|^2}-\frac{1}{4}\frac{1}{(1-\gamma)|A|}.
    \end{aligned}
\end{gather*}
Thus, combining \eqref{A4}, we have $\gamma=0$ and $|B|=0$.

\end{proof}

\subsection{The Pohozaev identity for bounded domain}
In order to prove the Pohozaev identity, we multiply \eqref{hEL} by $x \cdot \nabla w_L$ and we integrate by parts.
\begin{lem}\label{BP}
Assume that $d=1,2,3$ and $\mathcal{D}\subset\mathbb{R}^d$ be a  bounded domain. Let $w_L$ be a non-negative radially symmetric ground state of \eqref{min1} satisfy equation \eqref{hEL}, then we have
\begin{gather*}
    \begin{aligned}
    \frac{L^{-2}}{2(\rho|\mathcal{D}|)^{\frac{2}{d}}}\int_{\partial\mathcal{D}_0}|\nabla   w_L|^2(x\cdot\nu)&\,dS+\frac{(d-2)L^{-2}}{2(\rho|\mathcal{D}|)^{\frac{2}{d}}}\int_{\mathcal{D}_0}|\nabla w_L|^2\,dx\\
    &=\frac{d}{4}\int_{\mathcal{D}_0}|w_L|^4\,dx-\frac{d}{6}\int_{\mathcal{D}_0}|w_L|^6\,dx+\frac{d\mu_L}{2}\int_{\mathcal{D}_0}|w_L|^2\,dx,
    \end{aligned}
\end{gather*}
where $\mathcal{D}_0=\big\{x\in\mathbb{R}^d:\sqrt[d]{\rho |\mathcal{D}|}x\in \mathcal{D}\big\}$ and $\nu=\nu(x)$ denotes the outward unit normal vector to $\partial\mathcal{D}_0$. 
\end{lem}
\begin{pf}
Let $f(w_L)=w_L^3(x)-w_L^5(x)+\mu_L w_L(x)$ and
\begin{equation*}
F(w_L)=\int_0^{w_L} f(s)\,ds=\frac{1}{4}w_L^4-\frac{1}{6}w_L^6+\frac{\mu_L}{2}w_L^2.
\end{equation*}
It follows from \eqref{hEL} that
$$
0=\left(\frac{L^{-2}}{(\rho|\mathcal{D}|)^{\frac{2}{d}}}\Delta w_L+f(w_L)\right) x \cdot \nabla w_L.
$$
It is clear that
\begin{gather*}
\begin{aligned}
f(w_L) x \cdot \nabla w_L & =\operatorname{div}(x F(w_L))-d F(w_L), 
\end{aligned}
\end{gather*}
and
\begin{gather*}
\begin{aligned}
\Delta w_L x \cdot \nabla w_L & =\operatorname{div}(\nabla w_L x \cdot \nabla w_L)-|\nabla w_L|^2-x \cdot \nabla\left(\frac{|\nabla w_L|^2}{2}\right) \\
& =\operatorname{div}\left(\nabla w_L x \cdot \nabla w_L-x \frac{|\nabla w_L|^2}{2}\right)+\frac{d-2}{2}|\nabla w_L|^2.
\end{aligned}
\end{gather*}
Integrating by parts, we obtain
\begin{gather}\label{j1}
    \begin{aligned}
    &\quad\int_{\partial \mathcal{D}_0}\left(x F(w_L)+\frac{L^{-2}}{(\rho|\mathcal{D}|)^{\frac{2}{d}}}\nabla w_L x \cdot \nabla w_L-\frac{L^{-2}}{(\rho|\mathcal{D}|)^{\frac{2}{d}}} \frac{|\nabla w_L|^2}{2}x\right) \cdot \nu\,dS\\
    &=\int_{\mathcal{D}_0}\left[dF(w_L)-\frac{(d-2)}{2}\frac{L^{-2}}{(\rho|\mathcal{D}|)^{\frac{2}{d}}}|\nabla w_L|^2\right]\,dx.
    \end{aligned}
\end{gather}
On the other hand,  $w_L=0$ on $\partial \mathcal{D}_0$,  so that
$$
F(w_L)=0 \quad\text{and}\quad \nabla w_L=\nabla w_L \cdot \nu \nu.
$$
Combining \eqref{j1}, we obtain
$$
\frac{L^{-2}}{2(\rho|\mathcal{D}|)^{\frac{2}{d}}} \int_{\partial \mathcal{D}_0}|\nabla w_L|^2x\cdot\nu\,dS=\int_{\mathcal{D}_0}\left[dF(w_L)-\frac{(d-2)L^{-2}}{2(\rho|\mathcal{D}|)^{\frac{2}{d}}}|\nabla w_L|^2\right]\,dx.
$$
Thus,  we complete the proof of this lemma.
\end{pf}


\noindent\textbf{Acknowledgments} 
The authors declare that they have no conflict of interest. This work was supported by  the National Natural Science
Foundation of China (12301090).

\noindent\textbf{Data Availability Statement.} Data sharing not applicable to this article as no datasets were generatedor analyzed during the current study.





\end{document}